\documentclass{amsart}
\usepackage{amsmath, amssymb, amsthm, mathrsfs, color, comment}
\usepackage{tikz, here}
\usetikzlibrary{cd}
\usepackage{hyperref}
\usepackage{cleveref}

\newtheorem{dfn}{Definition}[section]
\newtheorem{thm}[dfn]{Theorem}

\newtheorem{lem}[dfn]{Lemma}
\newtheorem{cor}[dfn]{Corollary}
\newtheorem{prop}[dfn]{Proposition}
\newtheorem{rem}[dfn]{Remark}

\newtheorem{claim}{Claim}[dfn]
\newtheorem{subclaim}{Subclaim}[claim]

\crefname{dfn}{Definition}{Definitions}
\crefname{thm}{Theorem}{Theorems}
\crefname{fact}{Fact}{Facts}
\crefname{lem}{Lemma}{Lemmas}
\crefname{cor}{Corollary}{Corollaries}
\crefname{prop}{Proposition}{Propositions}
\crefname{rem}{Remark}{Remarks}
\crefname{notation}{Notation}{Notations}
\crefname{claim}{Claim}{Claims}
\crefname{subclaim}{Subclaim}{Subclaims}
\crefname{conj}{Conjecture}{Conjectures}
\crefname{ques}{Question}{Questions}
\crefname{section}{Section}{Sections}
\crefname{subsection}{Subsection}{Subsections}

\numberwithin{equation}{section}

\renewcommand{\subset}{\subseteq}
\newcommand{\power}{\wp}
\newcommand{\uphar}{\mathbin{\upharpoonright}}
\newcommand{\init}{\trianglelefteq}

\DeclareMathOperator{\dom}{dom}

\DeclareMathOperator{\crit}{crit}
\DeclareMathOperator{\lh}{lh}
\DeclareMathOperator{\ord}{Ord}
\DeclareMathOperator{\ran}{ran}
\DeclareMathOperator{\ult}{Ult}

\title{Long games just beyond fixed countable length}

\author{Takehiko Gappo}
\address{Takehiko Gappo, Institut f\"{u}r Diskrete Mathematik und Geometrie, TU Wien, Wiedner Hauptstra{\ss}e 8-10/104, 1040 Wien, Austria.}
\email{takehiko.gappo@tuwien.ac.at}

\author{Sandra M\"uller}
\address{Sandra M\"uller, Institut f\"ur Diskrete Mathematik und Geometrie, TU Wien, Wiedner Hauptstra{\ss}e 8-10/104, 1040 Wien, Austria.}
\email{sandra.mueller@tuwien.ac.at}

\date{January 7, 2026}

\begin{document}

    \maketitle

    \begin{abstract}
        We introduce a new type of game on natural numbers of variable countable length, which can be regarded as a diagonalization of all games of fixed countable length on natural numbers. Building on previous work by Trang and Woodin, we show that analytic determinacy of the game is equivalent to the existence of a sharp for a canonical inner model with a limit of Woodin cardinals $\lambda$ such that the order type of Woodin cardinals below $\lambda$ is $\lambda$.
    \end{abstract}

    \setcounter{tocdepth}{1}
    \tableofcontents

    \section{Introduction}

    Advances in inner model theory over the last 40 years have revealed deep connections between determinacy principles and large cardinals. While the Axiom of Determinacy ($\mathsf{AD}$) contradicts the Axiom of Choice, $\mathsf{ZFC}$ together with the existence of large cardinals implies the determinacy of definable sets of reals. For instance, Martin–Steel \cite{MS89} showed that the determinacy of all $\mathbf{\Pi}^1_{n+1}$ sets of reals follows from the existence of $n$ Woodin cardinals and a measurable cardinal above them. Neeman \cite{OptDet1} and Woodin \cite{MSW20} strengthened this result by establishing an equivalence between $\mathbf{\Pi}^1_{n+1}$ determinacy and the existence of sharps for canonical inner models with $n$ Woodin cardinals (see \cref{HMNW}).

    In order to obtain further equivalences between determinacy and the existence of sharps for inner models with large cardinals, it has become clear that \emph{long games}, whose lengths exceed $\omega$, play an essential role. Long games form a natural hierarchy ordered by their lengths, which can be classified into three categories:
    \begin{enumerate}
        \item Games of fixed countable length.
        \item Games of variable countable length, where the length of a run is always countable but depends on the players' moves.
        \item Games of length $\omega_1$.
    \end{enumerate}
    Building on Martin--Steel's proof of projective determinacy, Neeman developed general methods for proving long game determinacy from large cardinal assumptions in the late 1990s and early 2000s. His extensive work covers various long games from all three categories above, including games ending at the first admissible ordinal relative to the play \cite{adm}, games of continuously coded length \cite[Chapter 3]{Neemanbook}, games at a locally uncountable ordinal \cite[Chapter 7]{Neemanbook}, and games of open length $\omega_1$ \cite{Games_of_length_omega_1}. Although this is almost the entire list of long games studied so far, it is natural to expect that the hierarchy of long games is as rich as the hierarchy of large cardinals. In fact, Neeman's methods suggest how the complexity of winning strategies in long games mirrors the complexity of canonical inner models with large cardinals.

    This expectation has so far been confirmed only for games of fixed countable length, through level-by-level equivalence results between determinacy and the existence of sharps (see \cref{Trang-Woodin,Trang-Woodin2}). In the present paper, we introduce a new type of game of variable countable length and provide the first example of an equivalence between determinacy of variable countable length and a sharp for canonical inner models. Our work can be seen as a first step toward filling the substantial gaps in the hierarchy of games of variable countable length. In what follows, we introduce the necessary definitions and notation to state and motivate our main result.

    For any countable ordinal $\alpha<\omega_1$ and $A\subset\omega^\alpha$, the Gale--Stewart game $G_{\omega}(A)$ on $\mathbb{N}$ of length $\alpha$ with payoff set $A$ is defined as follows: two players take turns choosing natural numbers. Player I wins the game if the resulting sequence $\langle n_i\mid i<\alpha\rangle$ is in $A$; otherwise, Player II wins. We say that $G_{\alpha}(A)$, or simply $A$, is determined if one of the players has a winning strategy in the game $G_{\alpha}(A)$. The Axiom of Determinacy ($\mathsf{AD}$) is the statement that the games $G_\omega(A)$ are determined for all $A\subset\omega^\omega$.

    \begin{table}[ht]
        \centering
        \begin{tabular}{l||llllll}
            \,I  & $n_0$ &      & $n_2$ &     & $\cdots$ &       \\ \hline
            II &      & $n_1$ &      & $n_3$ &          & $\cdots$
        \end{tabular}
    \end{table}

    Now we precisely state the level-by-level equivalence results between determinacy of sets of reals and the existence of sharps.
    
    \begin{thm}[Harrington \cite{Har78} \& Martin \cite{meas_and_analytic_games} for $n=0$; Neeman \cite{OptDet1} \& Woodin \cite{MSW20} for $n\geq 1$]\label{HMNW}
    Let $n<\omega$. Then the following are equivalent over $\mathsf{ZFC}$:
    \begin{enumerate}
        \item For all $x\in\mathbb{R}$, there exists an $\omega_1$-iterable active $x$-premouse with $n$ many Woodin cardinals.
        \item All $\mathbf{\Pi}^1_{n+1}$ sets of reals are determined.
    \end{enumerate}
    \end{thm}

    \begin{thm}[Martin--Steel \cite{MS08} \& Woodin\footnotemark]\label{DetL(R)}
        The following are equivalent over $\mathsf{ZFC}$:
        \begin{enumerate}
            \item For all $x\in\mathbb{R}$, there exists an $\omega_1$-iterable active $x$-premouse with $n$ many Woodin cardinals.
            \item $\mathsf{AD}^{L(\mathbb{R})}$ holds and $\mathbb{R}^\sharp$ exists.
        \end{enumerate}
    \end{thm}
    \footnotetext{Woodin himself did not publish his proof, but one can find it in Trang's PhD thesis \cite{Tr15}. We also use the argument in \cref{mouse_from_DetSol}.}

    It is worth noting that long game determinacy can be added to both theorems. Indeed, the proof of \cref{HMNW} shows that the determinacy assumption in the theorem is equivalent to $\mathbf{\Pi}^1_1$ determinacy of games on $\mathbb{N}$ of length $\omega\cdot(n+1)$. Moreover, the clauses of \cref{DetL(R)} are equivalent to $\mathbf{\Pi}^1_1$ determinacy of games on $\mathbb{N}$ of length $\omega\cdot\omega$. Neeman adapted his determinacy proof of \cref{HMNW} for games of fixed countable length to obtain the following result.

    \begin{thm}[Neeman {\cite[Chapter 2]{Neemanbook}}]\label{Neeman}\leavevmode
        Let $1\leq\alpha<\omega_1$ and assume that there is an $\omega_1$-iterable active premouse with $\alpha$ many Woodin cardinals. Then the games $G_{\omega\cdot\alpha}(A)$ are determined for all $A\subset\omega^{\omega\cdot\alpha}$ that are ${<}\omega^2$-$\mathbf{\Pi}^1_1$ in the codes.\footnotemark
    \end{thm}\footnotetext{Using a bijection between $\omega$ and $\alpha$, one can find a bijection $\pi\colon\omega^{\omega\cdot\alpha}\to\omega^\omega$. Given a pointclass $\mathbf{\Gamma}$, we say that $A\subset\omega^{\omega\cdot\alpha}$ is \emph{$\mathbf{\Gamma}$ in the codes} if $\pi[A]$ is in $\mathbf{\Gamma}$.}

    The converse of this theorem is also true: For additively indecomposable $\alpha$, it is due to Trang--Woodin \cite{TrangThesis}. The additively decomposable cases are due to Aguilera--M\"uller \cite{AM20} and Aguilera--Gappo \cite{AG26}. Here we focus on explaining Trang--Woodin's work, which our main result heavily relies on.
    
    Let $\mathcal{F}$ be the club filter on $\wp_{\omega_1}(\mathbb{R})$. The model $L(\mathbb{R}, \mu) := L(\mathbb{R})[\mathcal{F}]$, where $\mu=\mathcal{F}\cap L(\mathbb{R})[\mathcal{F}]$, is called the Solovay model\footnote{Note that this Solovay model is not the same as Solovay's well-known model of $\mathsf{ZF}+\mathsf{DC}$, where all sets of reals have perfect set property, the Baire property, and are Lebesgue measurable.}, which first appeared in Solovay's work on $\mathbb{R}$-supercompactness of $\omega_1$ under determinacy \cite{Independence_DC}.  Trang and Woodin obtained the following equivalence based on the HOD analysis of the Solovay model.

    \begin{thm}[Neeman \cite{Neemanbook} \& Trang--Woodin \cite{TrangThesis}]\label{Trang-Woodin}
    The following are equivalent:
    \begin{enumerate}
        \item There is an $\omega_1$-iterable active premouse with $\omega^2$ Woodin cardinals.
        \item The games $G_{\omega^3}(A)$ are determined for all $A$ that are ${<}\omega^2$-$\mathbf{\Pi}^1_1$ in the codes.
        \item $L(\mathbb{R}, \mu)\models$``$\mathsf{AD} + \mu$ is an ultrafilter'' and a sharp for $L(\mathbb{R}, \mu)$ exists.
    \end{enumerate}
    \end{thm}
    
    To generalize \cref{Trang-Woodin} for the games of length $\omega^{\alpha}$ for all $\alpha<\omega_1$, Trang and Woodin used \emph{generalized Solovay models} introduced by Woodin. Using the ``club'' filters $\mathcal{F}_\alpha$ on $[\power_{\omega_1}(\mathbb{R})]^{\omega^\alpha}$ (see \cref{Sol_filters}), we write
    \[
    L(\mathbb{R}, \mu_{<\alpha}) := L(\mathbb{R})[\mathcal{F}_{<\alpha}],
    \]
    where
    \[
    \mathcal{F}_{<\alpha} = \{\langle\beta, Z\rangle\mid\beta<\alpha\land Z\in \mathcal{F}_{\beta}\}.
    \]
    Also, for each $\beta < \alpha$, let $\mu_{\beta} = \mathcal{F}_{\beta}\cap L(\mathbb{R}, \mu_{<\alpha})$. Note that $L(\mathbb{R}, \mu) = L(\mathbb{R}, \mu_{<1})$.
    
    \begin{thm}[Neeman \cite{Neemanbook} \& Trang--Woodin \cite{TrangThesis}]\label{Trang-Woodin2}
        For each $2\leq\alpha<\omega_1$, the following are equivalent:
        \begin{enumerate}
            \item For some (any) $x\in\mathbb{R}$ coding $\alpha$, there exists an $\omega_1$-iterable active $x$-premouse with $\omega^\alpha$ Woodin cardinals.
            \item The games $G_{\omega^{1+\alpha}}(A)$ are determined for all $A$ that are ${<}\omega^2$-$\mathbf{\Pi}^1_1$ in the codes.
            \item $L(\mathbb{R}, \mu_{<-1+\alpha})\models$``$\mathsf{AD} +\forall\beta<-1+\alpha\,(\mu_\beta$ is an ultrafilter on $[\power_{\omega_1}(\mathbb{R})]^{\omega^\beta})$'' and a sharp for $L(\mathbb{R}, \mu_{<-1+\alpha})$ exists.
        \end{enumerate}
    \end{thm}

    The ``next'' generalized Solovay model beyond $L(\mathbb{R}, \mu_{<\alpha})$, where $\alpha<\omega_1$, is $L(\mathbb{R}, \mu_{<\omega_1})$. In fact, Trang--Woodin obtained the following equiconsistency result even for this model.

    \begin{thm}[Trang--Woodin \cite{TrangThesis}]\label{TrangWoodinEquiconsistency}
        The following theories are equiconsistent:
        \begin{enumerate}
            \item $\mathsf{ZFC}\,+\,$There is a limit of Woodin cardinals $\lambda$ such that the order type of Woodin cardinals below $\lambda$ is $\lambda$.
            \item $\mathsf{ZFC}\,+\,L(\mathbb{R}, \mu_{<\omega_1})\models\text{``}\mathsf{AD} + \forall\alpha<\omega_1 (\mu_{\alpha}\text{ is an ultrafilter on }[\wp_{\omega_1}(\mathbb{R})]^{\omega^{\alpha}})$.''
        \end{enumerate}
    \end{thm}

    Our new type of long games, called \emph{diagonal games}, is corresponding to the large cardinal assumption of this equiconsistency result. To define games of variable countable length, we say that $L\subset\omega^{<\omega_1}$ is a \emph{length condition} if for any $\vec{x}\in\omega^{<\omega_1}$, there is $\alpha>0$ such that $\vec{x}\upharpoonright\alpha\in L$. For any length condition $L$ and any $A\subset L$, $G_L(A)$ is defined as the game on $\mathbb{N}$ in which players I and II take turns choosing natural numbers $n_i$ until they construct a sequence $\langle n_i\mid i<\alpha\rangle$ in $L$, and player I wins if and only if $\langle n_i\mid i<\alpha\rangle\in A$. A diagonal game $G_\Delta(A)$ is defined as $G_L(A)$ with length condition
    \[
    L = \{\langle n_i\mid i<\omega\cdot\alpha\rangle\mid\omega\cdot\alpha = \sup\{\omega\cdot(1+\|x_\beta\|_{\mathrm{WO}})\mid \beta<\alpha\}\},
    \]
    where $x_\beta = \langle n_i\mid \omega\cdot\beta\leq i < \omega\cdot(\beta+1)\rangle$ for each $\beta < \alpha$ and $\|x_\beta\|_{\mathrm{WO}}$ is the length of the well-order coded by $x_\beta$ or 0 if $x_\beta$ does not code a well-order. The diagonal game can be viewed as a ``diagonalization'' of games of fixed countable length. Our main theorem is the following.

    \begin{thm}\label{mainthm}
        The following are equivalent over $\mathsf{ZFC}$:
        \begin{enumerate}
            \item There is an $\omega_1$-iterable active premouse with a limit of Woodin cardinals $\lambda$ such that the order type of Woodin cardinals below $\lambda$ is $\lambda$.
            \item The games $G_{\Delta}(A)$ are determined for any $A\subset\mathbb{R}^{<\omega_1}$ that is ${<}\omega^2\mathchar`-\mathbf{\Pi}^1_1$ in the codes.
            \item $L(\mathbb{R}, \mu_{<\omega_1})\models\text{``}\mathsf{AD} + \forall\alpha<\omega_1 (\mu_{\alpha}\text{ is an ultrafilter on }[\wp_{\omega_1}(\mathbb{R})]^{\omega^{\alpha}})$,'' and a sharp for $L(\mathbb{R}, \mu_{<\omega_1})$ exists.
        \end{enumerate}
    \end{thm}

    The definition of the diagonal game provides a canonical way to define a new long game from a given $\omega_1$-sequence of long games. This diagonal operator on long games appears to play an important role in advancing through the hierarchy of games of variable countable lengths, potentially leading to further level-by-level equivalence theorems.

    Finally, we remark that the converse of aforementioned Neeman's determinacy theorems in \cite{Neemanbook,adm,Games_of_length_omega_1} are still open except for games of fixed countable length. We expect that any equivalence proof between mice and long game determinacy would always involve natural models of $\mathsf{AD}$ serving as ``generalized derived models'' of given mice. In our case, $L(\mathbb{R}, \mu_{<\omega_1})$ is such a model. It is, however, still unclear to us what types of $\mathsf{AD}$ models would play the same roles for the games studied by Neeman.

    \subsection*{Acknowledgments}
    This research was funded in whole by the Austrian Science Fund (FWF) [10.55776/Y1498, 10.55776/I6087, 10.55776/ESP5842424]. For the purpose of open access, the authors have applied a CC BY public copyright license to any Author Accepted Manuscript version arising from this submission. The first author thanks Andreas Lietz for helpful discussions on the topic of this paper.

    \section{Long game determinacy from a mouse}

    We need to adopt several conventions and notations. First, $\mathbb{R}$ always denotes the Baire space $\omega^\omega$. We confuse $\omega^{\omega\cdot\alpha}$ with $\mathbb{R}^{\alpha}$ and thus we confuse a sequence of natural numbers $\langle n_i\mid i<\omega\cdot\alpha\rangle$ with a sequence of reals $\langle x_\beta\mid \beta<\alpha\rangle$, where $x_\beta = \langle n_i\mid \omega\cdot\beta\leq i < \omega\cdot(\beta+1)\rangle$ for each $\beta < \alpha$. Following this convention, we often regard a sequence of reals $\langle x_\beta\mid\beta<\alpha\rangle$ as a run of $G_\Delta$, even though $G_\Delta$ is defined as a game on $\mathbb{N}$. When writing $G_\Delta$, we omit the payoff set of the game because it is not relevant for the general argument to follow.

    For any sequence $\langle x_\beta\mid\beta<\alpha\rangle\in\mathbb{R}^{<\omega_1}$, where $\alpha>0$, we define the ordinal $\eta_k^{\vec{x}}$ for $k<\omega$ as follows: Let $\eta_0^{\vec{x}} = 0$ and $\eta_1^{\vec{x}} = 1$. For each $1<k<\omega$, let
    \[
    \eta_{k+1}^{\vec{x}}=\sup\{1+\|x_{\beta}\|_{\mathrm{WO}}\mid \beta<\max\{\eta_k^{\vec{x}}, \alpha\}\}.
    \]
    Also, for $k<\omega$, we define $\theta^{\vec{x}}_k$ to be the unique ordinal such that $\eta_{k+1}^{\vec{x}}=\eta^{\vec{x}}_k + \theta^{\vec{x}}_k$. Note that $\eta_{k+1}^{\vec{x}}$ and $\theta_k^{\vec{x}}$ only depend on $\vec{x}\upharpoonright\eta_k^{\vec{x}}$. Then the following lemma holds by definition.

    \begin{lem}\label{trivial_lemma}
        Let $\vec{x} = \langle x_\beta\mid\beta<\alpha\rangle\in\mathbb{R}^{<\omega_1}$ be a complete run of $G_{\Delta}$.\footnote{We say a run of some game is complete if the run reaches a stopping point of the game.} Then either
        \begin{enumerate}
            \item for some $k < \omega$, $\eta_{k}^{\vec{x}} = \eta_{k+1}^{\vec{x}} = \alpha$, or
            \item for all $k<\omega$, $\eta_k^{\vec{x}}<\eta_{k+1}^{\vec{x}}$ and $\alpha=\sup\{\eta_k\mid k<\omega\}$.
        \end{enumerate}
    \end{lem}

    This lemma says that every complete run $\vec{x}$ of the game $G_\Delta$ can be canonically divided into at most $\omega$ many blocks $\vec{x}\upharpoonright[\eta_k^{\vec{x}}, \eta_{k+1}^{\vec{x}})$. Thus, we can think of $G_\Delta$ as the concatenation of at most $\omega$ many sub-games of fixed countable length. This is why $G_\Delta$ can be regarded as diagonalization of $\langle G_{\omega\cdot(1+\alpha)}\mid\alpha<\omega_1\rangle$. The crucial point is that even though the length of a run of the game $G_\Delta$ will only be determined during the run of the game, each sub-game determines the length of the next sub-game. This enables us to use Neeman's argument to prove the determinacy of $G_\Delta$ from mice.

    Moreover, \cref{trivial_lemma} allows us to code a run of $G_{\Delta}$ into a single real; each run $\vec{x}=\langle x_\beta\mid\beta<\alpha\rangle$ can be divided into at most $\omega$ many blocks $\vec{x}\upharpoonright[\eta_0^{\vec{x}}, \eta_{k+1}^{\vec{x}})$ as above. Each block can be coded into a real via a bijection between $\omega$ and $\theta_k^{\vec{x}}$, so if we fix some recursive injection of $\mathbb{R}^\omega$ into $\mathbb{R}$ beforehand, then we can code $\vec{x}$ into a single real. We denote the code of $\vec{x}$ by $\lceil\vec{x}\rceil$. For a given pointclass $\Gamma$, we say that $A\subset\mathbb{N}^{<\omega_1}\simeq\mathbb{R}^{<\omega_1}$ is in $\Gamma$ in the codes if there is $A^*\subset\mathbb{R}$ in $\Gamma$ such that $\vec{x}\in A \iff \lceil\vec{x}\rceil\in A^*$. The precise description of the coding system is not relevant for our arguments, but we require some kind of Lipschitz continuity in the sense that $\lceil\vec{x}\rceil\upharpoonright k$ only depends on $\vec{x}\upharpoonright\eta_k^{\vec{x}}$. Our goal in this section is to prove the following theorem:

    \begin{thm}\label{LGD_from_LC}
        Suppose that there exists a class model $M$ of $\mathsf{ZFC}$ and a cardinal $\lambda$ in $M$ such that
        \begin{itemize}
            \item $M = L(V_\lambda^M)$,
            \item $M$ is weakly iterable in the sense of \cite[Appendix A]{Neemanbook},
            \item $\lambda$ is the order type of Woodin cardinals below $\lambda$ in $M$, and
            \item $V_\lambda^M$ is countable in $V$.
        \end{itemize}
        Then the game $G_{\Delta}(C)$ is determined for every $C\subset\mathbb{R}^{<\omega_1}$ that is ${<}\omega^2\mathchar`-\mathbf{\Pi}^1_1$ in the codes.
    \end{thm}

    The proof is a modification of Neeman's proof of determinacy of the games of countable fixed length in \cite[Section 2D]{Neemanbook}. We only describe the outline of his argument and how to adapt it in our context.

    For the proof, we need to introduce several auxiliary games:
    \begin{enumerate}
        \item A genericity iteration game $\hat{G}_{\mathrm{fixed}}(M, \delta, \theta, \dot{A})$ and its mirror.
        \item Martin's open game $\mathcal{C}[\vec{y}]$ on ordinals for a given ${<}\omega^2\mathchar`-\mathbf{\Pi}^1_1$ set $C\subset\mathbb{R}^{<\omega_1}$ in the codes.
        \item Finite games $G_k(\vec{x}, P, \gamma)$ and their mirrors.
    \end{enumerate}

    The first game is literally the same as one in Neeman's book. The genericity iteration game $\hat{G}_{\mathrm{fixed}}(M, \delta, \theta, \dot{A})$ is defined for a transitive model $M$ of (a sufficiently large fragment of) $\mathsf{ZFC}$ with a Woodin cardinal $\delta$, a countable (in $V$) ordinal $\theta$, and a name $\dot{A}$ for a subset of $\mathbb{R}^{\theta}$ in $M^{\mathrm{Col}(\omega, \delta)}$. This game consists of $-1+\theta+1$ rounds, indexed 1 through $\theta$, each of which as length $\omega$. In each round $\xi$, the two players collaborate to produce a real $y_{\xi}$, while Player I plays a nice iteration tree $\mathcal{T}_{\xi}$ of length $\omega$ and Player II plays its cofinal branch $b_{\xi}$. Player I immediately wins if the final model along $b_\xi$ is ill-founded. Otherwise, Player I continues in the next round with a tree on the final model along $b_\xi$. Then the players obtain a stack $\langle\mathcal{T}_{\xi}, b_{\xi}\mid 1\leq\xi\leq\theta\rangle$ of nice iteration trees and branches, and its direct limit model $M_{\theta+1}$. Again, Player I wins if $M_{\theta+1}$ is ill-founded. In the case where $M_{\theta+1}$ is well-founded, Player I wins if and only if, letting $j\colon M\to M_{\theta+1}$ be the iteration embedding,
    \[\langle y_{\xi}\mid\xi<\theta\rangle\in j(\dot{A})[h]
    \]
    for some $M_{\theta+1}$-generic $h\subset\mathrm{Col}(\omega, j(\delta))$. The mirror version $\hat{H}_{\mathrm{fixed}}(M, \delta, \theta, \dot{A})$ is defined by switching the roles of player I and II. Please see \cite[Section 2A]{Neemanbook} for the exact definition. We use the following theorem from there.

    \begin{thm}\label{main_thm_on_games_of_fixed_length}
        Suppose that $\theta\geq 1$ is countable in $M$ and $V$ and that there are $-1+\theta$ many Woodin cardinals of $M$ below $\delta$ and $\delta$ itself is also a Woodin cardinal of $M$. Also, suppose that $V_{\delta+1}^M$ is a countable in $V$ and $g\subset\mathrm{Col}(\omega, \delta)$ is $M$-generic with $g\in V$. Let $\dot{A}$ and $\dot{B}$ be names for a subset of $\mathbb{R}^\theta$ in $M^{\mathrm{Col}(\omega, \delta)}$. Then at least one of the following holds:
        \begin{enumerate}
            \item Player I has a winning strategy in $\hat{G}_{\mathrm{fixed}}(M, \delta, \theta, \dot{A})$.
            \item Player II has a winning strategy in $\hat{G}_{\mathrm{fixed}}(M, \delta, \theta, \dot{B})$.
            \item In $M[g]$ there exists a sequence of reals $\vec{y}\in\mathbb{R}^\theta$ that is in neither $\dot{A}[g]$ nor $\dot{B}[g]$.
        \end{enumerate}
        Moreover, there are formulas $\phi_{\mathrm{I}}$ and $\phi_{\mathrm{II}}$ such that if $M\models\phi_{\mathrm{I}}[\delta, \theta, \dot{A}]$ then (1) holds and if $M\models\phi_{\mathrm{II}}[\delta, \theta, \dot{B}]$ then (2) holds.
    \end{thm}

    Now we introduce Martin's game $\mathcal{C}[\vec{y}]$ for every $\vec{y}\in\mathbb{R}^{<\omega_1}$ as in \cite[Section 2D]{Neemanbook}. Recall that $C\subset \mathbb{R}^{<\omega_1}$ is the given payoff set that is ${<}\omega^2\mathchar`-\mathbf{\Pi}^1_1$ in the codes. There is an ordinal $0<\alpha<\omega^2$ and a sequence $\langle C_{\xi}\mid\xi<\alpha\rangle$ of subsets of $\mathbb{R}^{<\omega_1}$ that are $\mathbf{\Pi}^1_1$ in the codes such that
    \[
    \vec{y}\in C \iff \text{the least $\xi$ such that $\xi=\alpha$ or $\vec{y}\in C_{\xi}$ is of the same parity a $\alpha$.}
    \]
    We write $\mathrm{LO}$ and $\mathrm{WO}$ for the set of linear orders on $\omega$ and the set of wellorders on $\omega$, respectively. For each $\xi<\alpha$, there is a Lipschitz continuous map $R_{\xi}\colon \mathbb{R}^{<\omega_1}\to\mathrm{LO}$ such that for all $\vec{y}\in\mathbb{R}^{<\omega_1}$,
    \[
    \vec{y}\in C_{\xi} \iff R_{\xi}[\vec{y}]\in\mathrm{WO}
    \]
    Here, the Lipschitz continuity of $R_{\xi}$ means that $R_{\xi}[\vec{y}]\uphar k+1$ only depends on $\vec{y}\uphar\eta_k^{\vec{y}}$. Without loss of generality, we may assume that $0$ is the maximal element of $R_{\xi}[\vec{y}]$.

    Recall that we assume $M=L(V_\lambda^M)$. Since $V_\lambda^M$ is countable in $V$, the sharp for $M$ exists. Let $m$ be such that $\alpha\in(\omega\cdot m, \omega\cdot (m+1)]$ and let $\langle u_i\mid i<m+1\rangle$ be an increasing enumeration of the first $m+1$ uniform Silver indiscernibles for $M$. Also, we fix an effective bijection $h\colon\omega\to\alpha\times\omega$ such that for any $\xi<\alpha$, letting $r_{\xi}=h^{-1}[\{\xi\}\times\omega]$,
    \begin{enumerate}
        \item $\mathrm{pr}_1\circ h$ is increasing on $r_{\xi}$, where $\mathrm{pr}_1(\langle\xi, i\rangle)=i$ for any $i<\omega$,
        \item if $\xi$ is even (resp.\ odd), then $r_{\xi}$ consists of even (resp.\ odd) natural numbers, and
        \item $\min(r_{\xi})<\min(r_{\xi+1})$.
    \end{enumerate}
    Then the game $\mathcal{C}[\vec{y}]$ is an open game on ordinals of length $\omega$ played as follows:
    \[
    \begin{array}{c|ccccc}
    \mathrm{I}    & \rho_0 &        & \rho_2 &        & \ldots\\ \hline
    \mathrm{II}   &        & \rho_1 &        & \rho_3 &
    \end{array}
    \]
    The rules of the games are: For any $\xi<\alpha$, letting $f_\xi\colon\omega\to\ord$ be given by $f_{\xi}(i)=\rho_j$ iff $h(j) = \langle \xi, i\rangle$,
    \begin{enumerate}
        \item $f_{\xi}$ embeds $R_{\xi}[\vec{y}]$ into ordinals, i.e.\
        \[
        \langle i, i'\rangle\in R_{\xi}[\vec{y}] \iff f_{\xi}(i)<f_{\xi}(i')
        \]
        for any $i, i'<\omega$,
        \item $f_{\xi}(0)<f_{\xi+1}(0)$, and
        \item if $\xi\in[\omega\cdot n, \omega\cdot(n+1))$, then $f_{\xi}(i)\in [u_n, u_{n+1})$.
    \end{enumerate}
    The first player to violate these rules will lose and II wins any infinite run. This definition induces a map $\vec{y}\mapsto\mathcal{C}[\vec{y}]$ mapping countable sequences $\vec{y}\in\mathbb{R}^{<\omega_1}$ to their game $\mathcal{C}[\vec{y}]$ that belongs to $M$ and we denote this map by $\mathcal{C}$. Notice that the map $\mathcal{C}$ is Lipschitz continuous in the sense that $\vec{y}\uphar\eta^{\vec{y}}_k$ is enough to determine the first $k+1$ rounds of $\mathcal{C}[\vec{y}]$.
    
    We say that a position $P$ of odd length in $\mathcal{C}[\vec{y}]$ is called \emph{good} (over $M$) if the ordinals given by Player II are Silver indiscernibles for $M$, and the ordinals given by Player I are definable in $M$ from the ordinals given by Player II and additional parameters in $V_\lambda^M\cup\{u_0, \ldots, u_m\}$. Also, a good position $P'$ is called an \emph{$M$-shift} of the good position $P$ if $P'$ and $P$ has the same length and each ordinals of $P'$ given by Player I are definable from the ordinals of $P'$ given by Player II in the same way as $P$. The following facts on $\mathcal{C}[\vec{y}]$ are mentioned as Facts 2D.10 and 2D.11 in \cite{Neemanbook}.

    \begin{lem}\label{Martin_game_facts}\leavevmode
        \begin{enumerate}
            \item Let $P$ be a good position of odd length in $\mathcal{C}[\vec{y}]$. Then there is an $M$-shift $P'$ of $P$ and a Silver indiscernible $\rho$ for $M$ such that $\rho$ is a legal move for II following $P'$ in $\mathcal{C}[\vec{y}]$.
            \item Let $\langle P_k\mid k<\omega\rangle$ be a sequence of good positions in $\mathcal{C}[\vec{y}]$ such that $P_{k+1}$ extends an $M$-shift of $P_k$ for any $k<\omega$. Then $\vec{y}\in C$.
        \end{enumerate}
    \end{lem}

    The game $G_k(\vec{x}, P, \gamma)$ is defined in a very similar way to the corresponding game in \cite{Neemanbook}. The only difference is the definition of $k$-sequences.

    For $\vec{x}\in\mathbb{R}^{<\omega_1}$ and $k<\omega$, let $\delta_k^{\vec{x}}$ be the $(\eta_k^{\vec{x}} + 1)$-th Woodin cardinal of $M$ in increasing order. For expository simplicity, we fix a sequence $\langle g_k\mid k<\omega\rangle$ such that each $g_k$ is $\mathrm{Col}(\omega, \delta_k^{\vec{x}})$-generic over $M[g_0 * \cdots * g_{k-1}]$. We also write $g^k = g_0 * \cdots * g_{k-1}$. Strictly speaking, we should use the forcing language instead of fixing generics.

    \begin{dfn}
        Let $k<\omega$. A triplet $\langle\vec{x}, P, \gamma\rangle\in M[g^k]$ is called a \emph{$k$-sequence over $M[g^k]$} if
        \begin{enumerate}
            \item $\gamma$ is an ordinal,
            \item $\vec{x}\in\mathbb{R}^{\eta_k^{\vec{x}}}$, and
            \item $P$ is a position after $k$ rounds in $\mathcal{C}[\vec{x}]$.
        \end{enumerate}
        Also, a triplet $\langle\vec{x}, P, \gamma\rangle\in M[g^k]$ is called an \emph{extended $k$-sequence over $M[g^k]$} if it satisfies the conditions (1), (2) and
        \begin{enumerate}
            \item[(3')] $P$ is a position after $k+1$ rounds in $\mathcal{C}[\vec{x}]$
        \end{enumerate}
        instead of (3).
    \end{dfn}

    Using this notion, we define in $M[g^k]$ a finite game $G_k(\vec{x}, P, \gamma)$ for any $k$-sequence $\langle\vec{x}, P, \gamma\rangle$ and a set of reals $A_k(\vec{x}, P^*, \gamma^*)$ for any extended $k$-sequence $\langle\vec{x}, P^*, \gamma^*\rangle$.

    \begin{dfn}\label{dfn:G_k}
        Let $k<\omega$ and let $\langle\vec{x}, P, \gamma\rangle$ be a $k$-sequence over $M[g^k]$. The game $G_k(\vec{x}, P, \gamma)$ is played as follows:
        \[
        \begin{array}{c|ccc}
            \mathrm{I}    &          & \rho_{2k} &             \\ \hline
            \mathrm{II}   & \gamma^* &           & \rho_{2k+1} 
        \end{array}
        \]
        There are two rules:
        \begin{itemize}
            \item II must play an ordinal $\gamma^*$ smaller than $\gamma$.
            \item $P^*:=P^{\frown}\langle\rho_{2k}, \rho_{2k+1}\rangle$ must be a legal position in $\mathcal{C}[\vec{x}]$.
        \end{itemize}
        Then I wins if and only if
        \[
        V_{\delta_{k+1}^{\vec{x}}}^M[g^k]\models\phi_{\mathrm{I}}[\delta_k^{\vec{x}}, \theta_k^{\vec{x}}, \dot{A}_k[g^k](\vec{x}, P^*, \gamma^*)].
        \]
        This definition induces a map $\langle \vec{x}, P, \gamma\rangle\mapsto G_k(\vec{x}, P, \gamma)$ that belongs to $M[g^k]$. Let $G_k=\dot{G}_k[g^k]$ be this map.
    \end{dfn}

    \begin{dfn}\label{dfn:A_k}
        Let $\langle\vec{x}, P^*, \gamma^*\rangle$ be an extended $k$-sequence over $M[g^k]$. Then a set $A_k(\vec{x}, P^*, \gamma^*)$ is a subset of $\mathbb{R}^{\theta_k^{\vec{x}}}$ in $M[g^k][g_k]$ defined by
        \[
        \vec{z}\in A_k(\vec{x}, P^*, \gamma^*) \iff M[g^k][g_k]\models\text{``I wins $G_{k+1}(\vec{x}^{\frown}\vec{z}, P^*, \gamma^*)$''}
        \]
        for any $\vec{z}\in\mathbb{R}^{\theta_k^{\vec{x}}}\cap M[g^k][g_k]$. Let $\dot{A}_k[g^k](\vec{x}, P^*, \gamma^*)\in M[g^k]$ be the canonical name for this set. This definition induces a map $\langle \vec{x}, P^*, \gamma^*\rangle\mapsto \dot{A}_k[g^k](\vec{x}, P^*, \gamma^*)$ that belongs to $M[g^k]$. Let $A_k=\dot{A}_k[g^k]$ be this map.
    \end{dfn}

    Note that \cref{dfn:G_k,dfn:A_k} should be simultaneously done by induction on $\gamma$ and $\gamma^*$, not by induction on $k$. The definition of $G_k(\vec{x}, P, \gamma)$ involves the set $\dot{A}_k[g^k](\vec{x}, P^*, \gamma^*)$, which in turn depend on $G_{k+1}(\vec{x}^{\frown}\vec{z}, P^*, \gamma^*)$ for $\gamma^*<\gamma$. One can find more explanation about these definitions in \cite[Section 2B(1)]{Neemanbook}.

    We can define the mirrored versions of these notions again by switching the roles of Players I and II: a finite game $H_k(\vec{x}, Q, \gamma)$ is defined in $M[g^k]$ for any mirrored $k$-sequence $\langle\vec{x}, Q, \gamma\rangle$ and a set of reals $B_k(\vec{x}, Q^*, \gamma^*)$ for any mirrored extended $k$-sequence $\langle\vec{x}, Q^*, \gamma^*\rangle$ by induction on the third coordinate of a mirrored (extended) $k$-sequence.

    Finally we let $\langle\gamma_L, \gamma_H\rangle$ be the least pair of local indiscernibles of $M$ relative to $u_m$, i.e.,
    \[
    V_{\gamma_L+\omega}^M\models\phi[\gamma_L, c_0, \ldots, c_{k-1}] \iff V_{\gamma_H+\omega}^M\models\phi[\gamma_H, c_0, \ldots, c_{k-1}]
    \]
    for any formula $\phi$ and any $c_0, \ldots, c_{k-1}\in V_{u_m+\omega}^M$. This completes our setup.

    \cref{LGD_from_LC} follows from the following three items:
    \begin{enumerate}
        \item If Player I has a winning strategy in $G_0(\emptyset, \emptyset, \gamma_L)$, then Player I has a winning strategy in $G_{\Delta}(C)$.
        \item If there is an ordinal $\gamma$ such that Player II has a winning strategy in $H_0(\emptyset, \emptyset, \gamma)$, then Player II has a winning strategy in $G_{\Delta}(C)$.
        \item One of the assumptions in (1) and (2) holds.
    \end{enumerate}
    The proofs of these items are almost literally the same as \cite[Section 2]{Neemanbook}. Therefore, We only go through the proof of (1) here and leave (2) and (3) to the reader. So we prove the following lemma:

    \begin{lem}
        If Player I has a winning strategy in $G_0(\emptyset, \emptyset, \gamma_L)$, then Player I has a winning strategy in $G_{\Delta}(C)$.
    \end{lem}
    \begin{proof}
        Fix an imaginary opponent for II in $G_{\Delta}(C)$ and a weak iteration strategy $\Gamma$ for $M$. We describe how to play against the imaginary opponent by constructing a run $\vec{x}=\langle x_{\xi}\mid\xi<\theta\rangle$ of $G_{\Delta}(C)$. At the start of stage $k<\omega$, we have the following objects:
        \renewcommand{\labelenumi}{(\alph{enumi})}
        \begin{enumerate}
            \item A sequence of reals $\vec{x}_k=\langle x_{\xi}\mid\xi<\eta_k^{\vec{x}_k}\rangle$,
            \item a $\Gamma$-iterate $M_k$ of $M=M_0$ and an iteration embedding $i_k\colon M\to M_k$,
            \item some $M_k$-generic $g^k = g_0 *\cdots * g_{k-1}\subset\mathrm{Col}(\omega, i_1(\delta_0))*\cdots *\mathrm{Col}(\omega, i_k(\delta_{k-1}))$, and
            \item a position $P_k$ after $k$ rounds in $i_k(\mathcal{C})[\vec{x}_k]$.
        \end{enumerate}
        Our induction hypotheses are:
        \renewcommand{\labelenumi}{(\roman{enumi})}
        \begin{enumerate}
            \item $\vec{x}_k \in M_k[g^k]$,
            \item $M_k[g^k]\models$``I wins $i_k(\dot{G}_k)[g^k](\vec{x}_k, P_k, \gamma_L)$,'' where $\langle\gamma_L, \gamma_H\rangle$ is the least pair of local indiscernibles of $M_k$ similar as before, and
            \item $P_k$ is a position in $i_k(\mathcal{C}_k)[\vec{x}_k]$ that is good over $M_k$.
        \end{enumerate}

        Let us start stage $k$ of the construction. By \cref{Martin_game_facts}(1), we can take an $M_k$-shift $P'_k$ of $P_k$ and a Silver indiscernible $\rho$ such that $P'_k{}^{\frown}\langle\rho\rangle$ is a legal move for Player I in $\mathcal{C}[\vec{x}]$. Then there is an elementary embedding $\tau\colon M_k[g^k]\to M_k[g^k]$ that fixes the uniform indiscernibles and $\tau(P_k)=P'_k$. Note that $\tau(\gamma_L)=\gamma_L$ and $\tau(\dot{G}_k)=\dot{G}_k$. By elementarity of $\tau$, I wins the game $i_k(\dot{G}_k)[g^k](\vec{x}_k, P'_k, \gamma_L)$. By local indiscernibility of $\langle\gamma_L, \gamma_H\rangle$, Player I wins the game $i_k(\dot{G}_k)[g^k](\vec{x}_k, P'_k, \gamma_H)$. Let $\sigma_k\in M_k[g^k]$ be a winning strategy for Player I witnessing this. We may assume that $\sigma_k$ is definable in $M_k[g^k]$ with parameters in $V_{i_k(\lambda)}^{M_k}[g^k]\cup\{u_0, \ldots, u_m, P'_k\}$.

        In the game $i_k(\dot{G}_k)[g^k](\vec{x}_k, P'_k, \gamma_H)$, let Player II play $\gamma_L$. Let $\rho_{2k}$ be Player I's response played by $\sigma_k$. Then play the Silver indiscernible $\rho$ as the next move, $\rho_{2k+1}$ for II. Let $\overline{P}_{k+1}=P'_k {}^{\frown}\langle\rho_{2k},\rho_{2k+1}\rangle$. Because $\sigma_k$ is a winning strategy for I, we have
        \[
        V_{i_k(\delta_{k+1}^{\vec{x}_k})}^{M_k}[g^k]\models\phi_{\mathrm{I}}[i_k(\delta_k^{\vec{x}_k}), \theta_k^{\vec{x}_k}, i_k(\dot{A}_k)[g^k](\vec{x}_k, \overline{P}_{k+1}, \gamma_L)]
        \]
        By \cref{main_thm_on_games_of_fixed_length}, Player I wins $\hat{G}_{\mathrm{fixed}}(M_k[g^k], i_k(\delta_k^{\vec{x}_k}), \theta_k^{\vec{x}_k}, \gamma_L)$. The difference between $V_{i_k(\delta^{\vec{x}_k}_{k+1})}^{M_k}[g^k]$ and $M_k[g^k]$ is not a problem here, because all relevant objects have small ranks below $i_k(\delta^{\vec{x}_k}_{k+1})$. Fix a winning strategy $\hat{\Sigma}_k$ witnessing this.
    
        Using this strategy, we can produce a complete run of the genericity iteration game $\hat{G}_{\mathrm{fixed}}(M_k[g^k], i_k(\delta_k^{\vec{x}_k}), \theta_k^{\vec{x}_k}, \gamma_L)$. This yields a sequence $\vec{y}_k:=\langle x_{\xi}\mid\xi\in[\eta_k^{\vec{x}_k},\eta_{k+1}^{\vec{x}_k})\rangle$ of reals and a weak iteration of $M_k$ of length $-1+\theta_k^{\vec{x}_k}+1$. Note that one can regard iteration trees on $M_k[g^k]$ as iteration trees on $M_k$. Let $M_{k+1}$ be the final model of this iteration and let $i_{k, k+1}\colon M_k\to M_{k+1}$ be the iteration embedding. Let $i_{k+1}=i_{k, k+1}\circ i_k$. Because $\hat{\Sigma}_k$ is winning for Player I, there is some $M_{k+1}[g^k]$-generic $g_k\subset\mathrm{Col}(\omega, i_{k+1}(\delta_k^{\vec{x}_k}))$ such that
        \[
        \vec{y}_k\in i_{k+1}(\dot{A}_k)[g^k][g_k](\vec{x}_k, i_{k, k+1}(\overline{P}_{k+1}), \gamma_L).
        \]
        Finally, let $\vec{x}_{k+1}=\vec{x}_k{}^{\frown}\vec{y}_k$ and let $P_{k+1}=i_{k, k+1}(\overline{P}_{k+1})$. Then the inductive hypothesis can be easily checked.

        Note that the construction guarantees that $P_{k+1}$ extends a shift of $i_{k, k+1}(P_k)$ for any $k<\omega$. Then we can show that the terminal position $\vec{x}=\langle x_{\xi}\mid \xi<\alpha\rangle$ is won by player I in $G_{\Delta}(C)$ as follows: Let $M_{\infty}$ be the direct limit of the models $M_k$. This is well-founded because of weak iterability of $M$. Let $P^*_k = i_{k, \infty}(P_k)$, where $i_{k, \infty}\colon M_k\to M_{\infty}$ is the direct limit map. Then each $P^*_k$ is good over $M_{\infty}$ and $P^*_{k+1}$ extends a shift of $P^*_k$ for any $k<\omega$. By \cref{Martin_game_facts}(2), this means $\vec{x}\in C$.
    \end{proof}

    As mentioned above, we leave it to the readers to check that Neeman's arguments also give (2) and (3), which completes the proof sketch of \cref{LGD_from_LC}.

    \section{Determinacy in the Solovay-type model from long game determinacy}\label{DetSol_from_LG}

    In this section, we introduce the Solovay-type model $L(\mathbb{R}, \mu_{<\omega_1})$ and prove $\mathsf{AD}$ in this model from long game determinacy. In particular, we show that (2) implies (3) in \cref{mainthm} in \cref{subsection:DetSol,subsection:SharpSol}. We also make a brief remark on $\mathsf{AD}^+$ and Mouse capturing in $L(\mathbb{R}, \mu_{<\omega_1})$ in \cref{subsection:MCSol}.
    
    \begin{dfn}[Woodin \cite{GenCM}]\label{Sol_filters}
        Let $1\leq\alpha<\omega_1$ and let $Z\subset[\power_{\omega_1}(\mathbb{R})]^{\omega^\alpha}$. The $\alpha$-Solovay game $G^{\mathrm{Sol}}_{\alpha}(Z)$ with payoff set $Z$ is a game on $\mathbb{R}$ of length $\omega^{1+\alpha}$ with typical as follows:
        \begin{table}[ht]
        \centering
            \begin{tabular}{l||llllll}
                \,I  & $x_0$ &      & $x_2$ &     & $\cdots$ &       \\ \hline
                II &      & $x_1$ &      & $x_3$ &          & $\cdots$
            \end{tabular}
        \end{table}
        
        \noindent
        Letting $\sigma_{\eta} = \{x_\xi\mid \xi < \omega\cdot(1+\eta)\}$ for each $\eta<\omega^\alpha$, Player II wins if and only if
        \[
        \langle\sigma_\eta\mid\eta<\omega^\alpha\rangle\in Z.
        \]
        Also, let $\mathcal{F}_{\alpha}$ be the set of all $Z\subset[\power_{\omega_1}(\mathbb{R})]^{\omega^{\alpha}}$ such that Player II has a winning strategy in the game $G^{\mathrm{Sol}}_\alpha(Z)$. It is not hard to see that $\mathcal{F}_\alpha$ is a filter, so we call it the \emph{$\alpha$-Solovay filter}.
    \end{dfn}

    \begin{dfn}[Woodin \cite{GenCM}]
        We define the $\Delta$-Solovay model as
        \[
        L(\mathbb{R}, \mu_{<\omega_1}) := L(\mathbb{R})[\mathcal{F}_{<\omega_1}],
        \]
        where
        \[
        \mathcal{F}_{<\omega_1} = \{\langle\alpha, Z\rangle\mid\alpha<\omega_1\land Z \in \mathcal{F}_{\alpha}\}.
        \]
        For each $\alpha < \omega_1$, we write
        \[
        \mu_{\alpha} = \mathcal{F}_{\alpha} \cap L(\mathbb{R}, \mu_{<\omega_1}).
        \]
    \end{dfn}

    Our target theorem in this section is the following.

    \begin{thm}\label{target_DetSol}
        Assume that the games $G_{\Delta}(C)$ are determined for any $C\subset\mathbb{R}^{<\omega_1}$ that is $\mathbf{\Pi}^1_1$ in the codes. Then
        \[
        L(\mathbb{R}, \mu_{<\omega_1})\models\mathsf{AD} +\mu_{\alpha}\text{ is an ultrafilter on $[\wp_{\omega_1}(\mathbb{R})]^{\omega^\alpha}$ for each $\alpha<\omega_1$}.
        \]
    \end{thm}

    To show this theorem, we introduce the following definition, which implicitly appeared in Trang's PhD thesis \cite{TrangThesis}.

    \begin{dfn}\label{g-derived_Solovay}
        Let $g = \langle\sigma_\xi\mid\xi < \omega^\alpha\rangle \in [\power_{\omega_1}(\mathbb{R})]^{<\omega_1}$. We define a sequence $\langle\eta_k^g\mid k<\omega\rangle$ of ordinals as follows: Let $\eta_0^g = 0$ and $\eta_1^g = 1$. For each $1 < k < \omega$, let
        \[
        \eta_{k+1}^g = \sup\{1 + \|x_\xi\|_{\mathrm{WO}}\mid\xi<\max\{\eta_k^g, \omega^\alpha\}\}.
        \]
        Also, let
        \[
        \mathbb{R}_g = \bigcup_{\xi<\omega^\alpha}g(\xi).
        \]
        For each $\beta < \alpha$, let $k_\beta$ be the least such that $\omega^\beta \leq \eta_{k_\beta+1}$. For any $m < \omega$, we define $f_\beta^m\colon\omega^\beta\to\omega^\alpha$ by
        \[
        f_\beta^m(\xi) = \eta_{k_\beta + m} + \xi
        \]
        for all $\xi < \omega^\beta$. Then we define $\mathcal{F}^g_\beta$ be the set of all $Z\subset[\power_{\omega_1}(\mathbb{R}_g)]^{\omega^\beta}$ such that for all but finitely many $m < \omega$,
        \[
        \langle g(f_\beta^m(\xi))\mid\xi < \omega^\beta\rangle \in Z.
        \]
        It is not hard to see that $\mathcal{F}^g_\beta$ is a filter. Finally, we define the \emph{$g$-derived $\Delta$-Solovay model} as
        \[
        L(\mathbb{R}_g, \mu_{<\alpha}^g) = L(\mathbb{R}_g)[\mathcal{F}_{<\alpha}^g],
        \]
        where
        \[
        \mathcal{F}_{<\alpha}^g = \{\langle\beta, Z\rangle\mid\beta<\alpha\land Z\in\mathcal{F}_{\beta}^g\}.
        \]
        For each $\beta < \alpha$, we write
        \[
        \mu_\beta^g = \mathcal{F}_\beta^g\cap L(\mathbb{R}_g)[\mathcal{F}_{<\alpha}^g].
        \]
    \end{dfn}

    \subsection{Determinacy in the Solovay-type model}\label{subsection:DetSol}

    We start proving \cref{target_DetSol} now. To define a game used in the proof, we need some preparation. The definition of this game is inspired by similar games in \cite{MaSt08, AM20}. These types of model existence games go back to Harvey Friedman \cite{Fri71}.
    
    Consider the language $\mathcal{L}=\{\in, \dot{X}, \dot{A}\}$ of set theory with two additional predicate symbols. We fix an $\mathcal{L}$-formula $\theta$ such that if $M=L_{\alpha}(\dot{X})[\dot{A}]$ for some ordinal $\alpha$ then
    \begin{enumerate}
        \item for any $x\in M$, there are $\beta<\alpha$ and $r\in X$ such that $M\models\theta[\beta, r, x]$, and
        \item for any $\beta<\alpha$ and $r\in X$, there is at most one $x\in M$ such that $M\models\theta[\beta, r, x]$.
    \end{enumerate}
    The formula $\theta$ uniformly defines an $\dot{X}^M$-parametrized family of well-orders the union of whose ranges is $M$. Such a formula can be found analogously to the existence of a definable well-order in $L$.
    
    Also, let $\mathcal{L}^+ = \mathcal{L}\cup\{\dot{x}_i\mid i<\omega\}$ be the language obtained by adding constant symbols $\dot{x}_i$ for every $i<\omega$ to $\mathcal{L}$. We fix a recursive enumeration $\langle\phi_i\mid i<\omega\rangle$ of all $\mathcal{L}^+$-sentences such that the symbol $\dot{x}_i$ does not appear in $\phi_k$ if $k\leq i$. Also, we fix recursive bijections $m$ and $n$ assigning an odd number $>1$ to each $\mathcal{L}^+$-formula $\phi$ such that $m$ and $n$ have disjoint recursive ranges and for every $\phi$, $m(\phi)$ and $n(\phi)$ is larger than $\max\{i\mid\dot{x}_i\text{ occurs in }\phi\}$.

    Let $\phi$ be an $\mathcal{L}$-sentence and let $N$ be an $\mathcal{L}$-structure with $N = L_{\alpha}(\dot{X}^N)[\dot{A}^N]$ for some ordinal $\alpha$. We say that an $\mathcal{L}$-structure $N$ is a \emph{minimal $\phi$-witness} if $N$ satisfies $\phi$ but no proper $L$-initial segment of $N$ does.

    Now we define a game $G_{\phi, \psi}$ for a fixed $\mathcal{L}$-sentence $\phi$ and an $\mathcal{L}$-formula $\psi(u, v_0, v_1, v_2)$ as follows. A typical run of the game looks as follows.
    \[
    \begin{array}{c|ccccccccc}
    \mathrm{I}   & z_0, z_1, a_0 &     & \cdots & v_0, x_0 &     & \cdots & v_1, x_{\omega\cdot\eta_0} &                         & \cdots\\ \hline
    \mathrm{II}  &     & a_1 & \cdots &          & x_1 & \cdots &                            & x_{\omega\cdot\eta_0+1} & \cdots
    \end{array}
    \]
    The game is naturally divided into at most $\omega$ rounds.
    \begin{itemize}
        \item \textsc{Round $0$} has length $\omega$. Player I plays two reals $z_0, z_1$ and then two players take turns playing natural numbers $a_0, a_1, \ldots$ to construct a single real $y = \langle a_n\mid n<\omega\rangle\in\mathbb{R}$. We let $\eta_0 = 1 + \|z_1\|_{\mathrm{WO}}$.
        \item \textsc{Round $1$} has length $\omega\cdot\eta_0$. Player I plays a digit $v_0\in\{0, 1\}$ and then the two players take turns playing reals $x_\xi$, where $\xi<\omega\cdot\eta_0$. Let
        \[
        \eta_1 = \sup\{1 + \|x_{\xi}\|_{\mathrm{WO}}\mid \xi < \omega\cdot\eta_0\}.
        \]
        If $\eta_1 \leq \eta_0$, then we stop the game and Player I wins. Otherwise, we proceed to the next round.
        \item Suppose that we reach \textsc{Round $n+1$}, where $n > 0$. The round has length $\omega\cdot\eta_n$ and Player I starts with playing a digit $v_n\in\{0, 1\}$. Then the two players take turns playing reals $x_\xi$, where $\omega\cdot\eta_{n-1}\leq\xi<\omega\cdot\eta_n$. Let
        \[
        \eta_{n+1} = \sup\{1 + \|x_{\xi}\|_{\mathrm{WO}}\mid \xi < \eta_n\}.
        \]
        If $\eta_{n+1} = \eta_n$, then we stop the game and Player I wins. Otherwise, we proceed to the next round.
    \end{itemize}
    Suppose that $\omega$ many rounds have been played. Then we stop the game and decide a winner of the game as follows. Let $\lambda = \sup\{\eta_n\mid n<\omega\}$ and let $g\colon\lambda\to\wp_{\omega_1}(\mathbb{R})$ be defined by
    \[
    g(\eta) = \{x_\xi\mid\xi < \omega\cdot (1+\eta)\}.
    \]
    Also, let $T$ be the $\mathcal{L}^+$-theory that consists of the sentences $\phi_i$ such that $v_i = 1$. Then Player I wins if and only if
    \begin{enumerate}
        \item $T$ contains the sentences ``$\dot{x}_k\in\mathbb{R}$'' and ``$\dot{x}_k(m) = n$'' whenever $x_{c(k)}(m)=n$, where $c\colon\omega\to\lambda$ is a bijection that is canonically induced from the run as explained in the paragraph before \cref{LGD_from_LC} when coding a run of the game into a single real.
        \item $T$ satisfies the witness property in the sense that for every $\mathcal{L}^+$-formula $\varphi$ with one free variable, $T$ contains the following $\mathcal{L}^+$-sentences:
        \begin{gather*}
            \exists x\,\varphi(x)\rightarrow\exists x\,\exists\alpha\,(\varphi(x)\land\theta(\alpha, \dot{x}_{m(\phi)}, x)),\\
            \exists x\,(\varphi(x)\land x\in\dot{X})\to\varphi(\dot{x}_{n(\varphi)}).
        \end{gather*}
        \item $T$ is a complete, consistent theory such that for every countable model $M$ of $T$ and every model $N^*$ which is the definable closure of $\{x_{\eta}\mid\eta<\omega\cdot\alpha\}$ in $M$, $N^*$ is well-founded and letting $N$ be the transitive collapse of the model $N^*$,
        \begin{enumerate}
            \item $N$ is an $L$-initial segment of the $g$-derived $\Delta$-Solovay model,
            \item $N$ is a minimal $\phi$-witness, and
            \item $N\models\psi[\langle x_\xi\mid\xi < \omega\cdot\eta_0\rangle, y, z_0, z_1]$.
        \end{enumerate}
    \end{enumerate}
    Note that $\mathbf{\Pi}^1_1$-determinacy of $G_\Delta$ implies the determinacy of $G_{\phi, \psi}$ for any $\phi$ and $\psi$.

    \begin{thm}\label{DetSol_from_LongGame}
        Assume that the games $G_{\Delta}(C)$ are determined for any $C\subset\mathbb{R}^{<\omega_1}$ that is $\mathbf{\Pi}^1_1$ in the codes. Then
        \[
        L(\mathbb{R}, \mu_{<\omega_1})\models\mathsf{AD}\land\forall\alpha<\omega_1\,(\mu_{\alpha}\text{ is an ultrafilter on }[\power_{\omega_1}(\mathbb{R})]^{\omega^\alpha}).
        \]
    \end{thm}
    \begin{proof}
        Suppose toward a contradiction that there is an ordinal $\gamma$ such that
        \[
        L_{\gamma + 1}(\mathbb{R}, \mu_{<\omega_1})\not\models\mathsf{AD}\land\forall\alpha<\omega_1\,(\mu_{\alpha}\text{ is an ultrafilter on }[\power_{\omega_1}(\mathbb{R})]^{\omega^\alpha}).
        \]
        and let $\gamma^*$ be the least such $\gamma$.

        \begin{claim}\label{ad_does_not_fail}
            $L_{\gamma^* + 1}(\mathbb{R}, \mu_{<\omega_1})\models\mathsf{AD}$.
        \end{claim}
        \begin{proof}
            Let $\phi$ be the sentence ``$\dot{X} = \mathbb{R} \land \neg\mathsf{AD}$'' and let $\psi(y, z_0, z_1)$ be the formula asserting that there is a non-determined set of reals definable from $z_0$ and letting $Z$ be the least such set in the well-order given by $\theta$, $y\in Z$. By our determinacy assumption, the game $G_{\phi, \psi}$ is determined.

            \noindent\textbf{Case 1: Player I has a winning strategy $\tau$ in $G_{\phi, \psi}$.} Take a continuous $\in$-chain $\langle H_{\xi}\mid\xi<\omega_1\rangle$ of countable elementary submodels of $V_\chi$, where $\chi$ is a large enough regular cardinal, such that $\tau\in H_0$. For each $\xi < \omega_1$, let
            \[
            \mathbb{R}_{\xi} = \mathbb{R} \cap H_{\xi},
            \]
            let $M_{\xi}$ be the transitive collapse of $H_{\xi}$, and let $\pi_{\xi}\colon M_{\xi}\to V_{\chi}$ be the uncollapse map. Let $\alpha$ be the least ordinal such that
            \[
            \alpha = \omega_1 \cap H_{\alpha}.
            \]
            Now consider a run of $G_{\phi, \psi}$ where Player I follows her winning strategy $\tau$ and Player II plays reals so that $\mathbb{R}_\xi = \{x_{\omega\cdot\xi + 2i + 1}\mid i < \omega\}$. Since $\tau \in M_\xi$, Player I's moves $x_{\omega\cdot\xi + 2i}$ are all in $M_\xi$. The game will stop exactly at length $\alpha$ and we have
            \[
            g = \langle\mathbb{R}_\xi\mid\xi<\alpha\rangle.
            \]
            Then we consider the $g$-derived $\Delta$-Solovay model as introduced in \cref{g-derived_Solovay} and obtain the following.

            \begin{subclaim}\label{play_gives_Solovay_model}
                $\pi_{\alpha}^{-1}(L_{\gamma^* + 1}(\mathbb{R}, \mu_{<\omega_1})) = L_{\pi_{\alpha}^{-1}(\gamma^* + 1)}(\mathbb{R}_g, \mu^g_{<\alpha})$.
            \end{subclaim}
            \begin{proof}
                Denote the $\Delta$-Solovay model in $M_\alpha$ by $L(\mathbb{R}_\alpha, \overline{\mu}_{<\alpha})$ and let $\overline{\mathcal{F}}_\beta = \pi_{\alpha}^{-1}(\mathcal{F}_\beta)$ for each $\beta < \alpha$. By induction on $\xi\leq\pi_{\alpha}^{-1}(\gamma^*)$, we will show that
                \[
                L_{\xi}(\mathbb{R}_\alpha, \overline{\mu}_{<\alpha}) = L_{\xi}(\mathbb{R}_g, \mu^g_{<\alpha}).
                \]
                This is trivially true for $\xi = 0$ and limit ordinals $\xi$. Consider the case that $\xi = \xi' + 1$. Then it is enough to show that for any $\beta < \alpha$,
                \begin{equation}\label{induction_on_height}
                    \overline{\mathcal{F}}_{\beta}\cap L_{\xi'}(\mathbb{R}_\alpha,\overline{\mu}_{<\alpha}) = \mathcal{F}_{\beta}^g\cap L_{\xi'}(\mathbb{R}_g, \mu_{<\alpha}^g).
                \end{equation}
            
                First, we show the $\subset$ direction of (\ref{induction_on_height}). Let $Z\in\overline{\mathcal{F}}_{\beta}\cap L_{\xi'}(\mathbb{R}_\alpha, \overline{\mu}_{<\alpha})$ witnessed by Player II's strategy $F\colon\mathbb{R}_{\alpha}^{<\omega^{\beta}}\to\mathbb{R}_{\alpha}$ for $G^{\mathrm{Sol}}_\beta(Z)$ in $M_{\alpha}$. Note that $F = \pi_{\alpha}(F)\in H_\alpha$. Find $m<\omega$ such that $\eta_m > \omega^\beta$ and $F\in H_{\eta_m}$. Then for any $\eta\geq\eta_m$, $F = \pi_{\eta}^{-1}(F) \in M_\eta$ and thus $M_\eta$ is closed under $F$, i.e.,
                \begin{equation}\label{closure_under_F}
                    F[\mathbb{R}_{\eta}^{<\omega^\beta}]\subset\mathbb{R}_\eta.
                \end{equation}
                Now fix any $n\geq m$ and consider a run of $G^{\mathrm{Sol}}_\beta(Z)$ in $M_\alpha$, where Player I plays so that
                \[
                \mathbb{R}_{f^n_\beta(\xi)} = \{x_\eta\mid\eta < \omega\cdot(1+\xi)\land\eta\text{ is even}\}
                \]
                for all $\xi < \omega^\beta$ and Player II plays according to $F$. By (\ref{closure_under_F}), it follows that, letting $\langle\sigma_\xi\mid\xi<\omega^\beta\rangle\in[\power_{\omega_1}(\mathbb{R}_\alpha)]^{\omega^\beta}$ be the resulting sequence of the run,
                \[
                \mathbb{R}_{f^n_\beta(\xi)} = \sigma_\xi,
                \]
                for all $\xi < \omega^\beta$. Because $F$ is Player II's winning strategy for $G^{\mathrm{Sol}}_\beta(Z)$ in $M_\alpha$,
                \[
                \langle g(f^n_\beta(\xi))\mid \xi < \omega^\beta\rangle = \langle\sigma_\xi\mid\xi<\omega^\beta\rangle \in Z.
                \]
                It follows that $Z\in\mathcal{F}_\beta^g$.
            
                Now we show the $\supseteq$ direction of (\ref{induction_on_height}). Let $Z\in L_{\xi'}(\mathbb{R}_\alpha, \overline{\mu}_{<\alpha})\setminus\overline{\mathcal{F}}_{\beta}$. Because of the minimality of $\gamma^*$ and elementarity of $\pi_\alpha$, $\overline{\mathcal{F}}_\beta$ is an ultrafilter over
                \[
                L_{\xi'}(\mathbb{R}_\alpha, \overline{\mu}_{<\alpha}) = L_{\xi'}(\mathbb{R}_g, \mu^g_{<\alpha}).
                \]
                Here, the equation is the induction hypothesis for $\xi'$. It follows that the complement of $Z$ is in $\overline{\mathcal{F}}_\beta\cap L_{\xi'}(\mathbb{R}_\alpha, \overline{\mu}_{<\alpha})$. The argument in the last paragraph shows that the complement of $Z$ is in $\mathcal{F}_\beta^g$ and thus $Z\notin\mathcal{F}_\beta^g$.
            \end{proof}

            Suppose toward a contradiction that \cref{ad_does_not_fail} fails. By \cref{play_gives_Solovay_model}, the minimal $\phi$-witness $N$ played by Player I is always $\pi_{\alpha}^{-1}(L_{\gamma^* +1}(\mathbb{R}, \mu_{<\omega_1}))$. Because $\tau$ is a winning strategy for $G_{\phi, \psi}$, letting $z_0$ be the first move by $\tau$, there is a non-determined set of reals definable from $z_0$ in $N$. Let $Z$ be the least such one in the well-order given by $\theta$. By varying Player II's play $\langle a_{2n+1}\mid n<\omega\rangle$, we can define Player I's strategy $\tau^*\in M_\alpha$ in the Gale--Stewart game on $\mathbb{N}$ (of length $\omega$). If Player I follows $\tau^*$ in \textsc{Round 0} in $G_{\phi, \psi}$, the players will end up producing $y \in Z$ in \textsc{Round 0}. This witnesses that $Z$ is determined in $N$, which is a contradiction.

            \noindent\noindent\textbf{Case 2: Player II has a winning strategy $\tau$ in $G_{\phi, \psi}$.} The argument is almost symmetric to the first case, but we sketch it for reader's convenience. Again, we take a continuous $\in$-chain $\langle H_\xi\mid\xi<\omega_1\rangle$ of countable elementary submodels of $V_\chi$, where $\chi$ is an enough large regular cardinal, such that $\tau\in H_0$. We use the same notation as before and let $\alpha$ be the least ordinal such that $\alpha = \omega_1\cap H_\alpha$. Consider a run of $G_{\phi, \psi}$, where Player II follows $\tau$ and Player I plays reals in \textsc{Rounds} 1 to $\omega$ so that $\mathbb{R}_\xi = \{x_{\omega\cdot\xi + 2i + 1}\mid i<\omega\}$. As before, the game will stop exactly at length $\alpha$ and we have $g = \langle\mathbb{R}_\xi\mid\xi<\alpha\rangle$. Again, we have
            \[
            N: = \pi_{\alpha}^{-1}(L_{\gamma^*}(\mathbb{R}, \mu_{<\omega_1})) = L_{\pi_{\alpha}^{-1}(\gamma^*)}(\mathbb{R}_g, \mu^g_{<\alpha}).
            \]
            Assume toward a contradiction that \cref{ad_does_not_fail} fails. Then there is a non-determined set in $N$. So Player I can pick a real $z_0$ such that there is a non-determined set of reals definable from $z_0$. Let $Z$ be the least such set in the well-order given by $\theta$. Player I also plays the theory of $N$. By varying Player I's play $\langle a_{2n}\mid n<\omega\rangle$, we can define Player II's strategy $\tau^*\in M_\alpha$ in the Gale--Stewart game on $\mathbb{N}$ (of length $\omega$). If Player II follows $\tau^*$ in \textsc{Round 0} in $G_{\phi, \psi}$, the players will end up producing $y \notin Z$ in \textsc{Round 0}. This witnesses that $Z$ is determined in $N$, which is a contradiction.            
        \end{proof}

        \begin{claim}\label{ultrafilterness_does_not_fail}
            $L_{\gamma^* +1}(\mathbb{R}, \mu_{<\omega_1})\models\forall\alpha<\omega_1\,(\mu_{\alpha}\text{ is an ultrafilter on }[\power_{\omega_1}(\mathbb{R})]^{\omega^\alpha})$.
        \end{claim}
        \begin{proof}
            Let $\phi$ be the sentence
            \[
            \dot{X} = \mathbb{R}\land\exists\alpha< \omega_1\,(\mu_\alpha\text{ is not an ultrafilter on }[\power_{\omega_1}(\mathbb{R})])
            \]
            and let $\psi$ be the formula asserting that there is a non-measured set by $(\dot{A})_{\|z_1\|_{\mathrm{WO}}}$ definable from $z_0$ and letting $Z$ be the least such set in the well-order given by $\theta$, $\langle x_{\xi}\mid\xi<\omega\cdot\|z_1\|_{\mathrm{WO}}\rangle\in Z$. By our determinacy assumption, the game $G_{\phi, \psi}$ is determined.

            \noindent\textbf{Case 1: Player I has a winning strategy $\tau$ in $G_{\phi, \psi}$.} As before, we take a continuous $\in$-chain $\langle H_\xi\mid\xi<\omega_1\rangle$ of countable elementary submodels of $V_\theta$, where $\theta$ is an enough large regular cardinal, such that $\tau\in H_0$. For each $\xi<\omega_1$, let $\mathbb{R}_\xi, M_\xi, \pi_\xi$ as before. Also, let $\alpha$ be the least such that $\alpha = \omega_1\cap H_\alpha$.

            Now consider a run of $G_{\phi, \psi}$ where Player I follows $\tau$ and Player II plays reals so that $\mathbb{R}_\xi = \{x_{\omega\cdot\xi + 2i + 1}\mid i < \omega\}$. Since $\tau \in M_\xi$, Player I's moves $x_{\omega\cdot\xi + 2i}$ are all in $M_\xi$. The game will stop exactly at length $\alpha$ and we have
            \[
            g = \langle\mathbb{R}_\xi\mid\xi<\alpha\rangle.
            \]
            As in the proof of \cref{play_gives_Solovay_model}, we have
            \[
            \pi_{\alpha}^{-1}(L_{\gamma^* +1}(\mathbb{R}, \mu_{<\omega_1})) = L_{\pi_{\alpha}^{-1}(\gamma^*) +1}(\mathbb{R}_g, \mu^g_{<\alpha}).
            \]
            Now we need another observation. For $n<\omega$, we define $g^{(n)}\colon\alpha\to\power_{\omega_1}(\mathbb{R}_\alpha)$ by
            \[
            g(\eta_m + \xi) = g(\eta_{m+n}+\xi),
            \]
            where $m<\omega$ and $\xi\geq 0$ is such that $\eta_m + \xi < \eta_{m+1}$. Note that for any $\beta<\omega$ and $m, n<\omega$,
            \[
            \langle g^{(n)}(f^m_\beta(\xi))\mid\xi < \omega^\beta\rangle = \langle g(f^{m+n}_\beta(\xi))\mid\xi<\omega^\beta\rangle,
            \]
            because $k_{f^m_\beta(\xi)} = k_\beta + m$ holds. Recall that $k_\beta$ is defined as the least $k<\omega$ such that $\beta\leq\eta_{k+1}$). Therefore, we have
            \[
            \mathcal{F}^{g^{(n)}}_\beta = \mathcal{F}^{g}_\beta
            \]
            and
            \[
            L(\mathbb{R}_g, \mu^g_{<\alpha}) = L(\mathbb{R}_{g^{(n)}}, \mu^g_{<\alpha}).
            \]

            Now suppose that \cref{ultrafilterness_does_not_fail} does not hold. We know that the minimal $\phi$-witness $N$ played by Player I is always $\pi_{\alpha}^{-1}(L_{\gamma^* +1}(\mathbb{R}, \mu_{<\omega_1}))$. Because $\tau$ is a winning strategy for $G_{\phi, \psi}$, letting $z_0, z_1$ be the first two moves by $\tau$, there is a non-measured set by $\mu_{\|z_1\|_{\mathrm{WO}}}$ definable from $z_0$ in $N$. Let $Z$ be the least such set in the well-order given by $\theta$. By varying Player II's play, one can show
            \[
            \langle g^{(n)}(f^0_\beta(\xi))\mid\xi<\omega^\beta\rangle\in Z
            \]
            for all $n<\omega$. It follows that $Z\in\mathcal{F}^{g^{(n)}}_\beta = \mathcal{F}^g_\beta$. This is a contradiction.

            \noindent\textbf{Case 2: Player II has a winning strategy $\tau$ in $G_{\phi, \psi}$.} The argument is almost symmetric in a way that we already explained in the proof of \cref{ad_does_not_fail}, so we leave the details to the reader.
        \end{proof}

        \cref{ad_does_not_fail,ultrafilterness_does_not_fail} contradict the definition of $\gamma^*$, which completes the proof of \cref{DetSol_from_LongGame}.
    \end{proof}

    \subsection{Sharp for the Solovay-type model}\label{subsection:SharpSol}

    In this subsection, we prove the existence of a sharp for the model $L(\mathbb{R}, \mu_{<\omega_1})$ from the ${<}\omega^2\mathchar`-\mathbf{\Pi}^1_1$ determinacy of $G_\Delta$.

    \begin{dfn}\label{dfn:sharp}
        We say that the sharp for $L(\mathbb{R}, \mu_{<\omega_1})$ exists if there is a nontrivial elementary embedding
        \[
        j\colon (L(\mathbb{R}, \mu_{<\omega_1}), \in, \mu_{<\omega_1}) \to (L(\mathbb{R}, \mu_{<\omega_1}), \in, \mu_{<\omega_1})
        \]
        such that $\crit(j)>\Theta^{L(\mathbb{R}, \mu_{<\omega_1})}$.
    \end{dfn}

    One can characterize the sharp for $L(\mathbb{R}, \mu_{<\omega_1})$ in various standard ways. Consequently, generic absoluteness of the sharp for $L(\mathbb{R}, \mu_{<\omega_1})$ holds, i.e., if the sharp for $L(\mathbb{R}, \mu_{<\omega_1})$ exists in some generic extension (by a set-sized forcing) then so it does in the ground model.

    To prove the following theorem, we use the same argument as in \cite[Subsection 2.3.2]{TrangThesis}, which is originally due to Woodin. One interesting aspect of the proof is the use of Jensen's model existence theorem. We recommend to see Chapter 2 of \cite{JensenSubcomplete} or \cite{ChanNotes}.\footnote{The same result was also proved in \cite[Theorem 2.2]{Fri73}. It seems that they independently got the same result almost at the same time. According to the preface of \cite{JensenAdmissible}, Jensen gave a set of lectures on the topic at the Rockefeller University in 1969. According to the preface of \cite{SummerSchoolBook73}, Harvey Friedman gave a lecture on it at the Summer School in Mathematical Logic that was held in Cambridge in 1971.} The first author thanks Andreas Lietz for his help to understand the argument in \cite{TrangThesis}.

    \begin{thm}\label{SharpSol_from_LongGame}
        Suppose that the games $G_{\Delta}(C)$ are determined for any $C\subset\mathbb{R}^{<\omega_1}$ that is $\mathbf{\Pi}^1_1$ in the codes. Then the sharp for the $\Delta$-Solovay model $L(\mathbb{R}, \mu_{<\omega_1})$ exists.
    \end{thm}
    \begin{proof}
        In the proof below, ill-founded models of set theory will play essential roles. We assume that the well-founded part of such a model is always transitive.
        
        We define the following two-player game $G^\sharp$, which is a variant of the game $G_{\phi, \psi}$ introduced in the last subsection. A typical run of the game looks as follows.
        \[
        \begin{array}{c|cccccc}
            \mathrm{I}   & m_0, x_0 &          & \cdots & m_1, x_{\omega\cdot\eta_0} &                              & \cdots\\ \hline
            \mathrm{II}  &          & v_0, x_1 & \cdots &                            & v_1, x_{\omega\cdot\eta_0+1} & \cdots
        \end{array}
        \]
        The game is divided into at most $\omega$ many rounds. In each round, Player I and II alternately choose reals $x_\beta$. In addition, at the beginning of each round, Player I plays a natural number $m_i\in\omega$ and Player II plays a digit $v_i\in\{0, 1\}$. The length of the game is determined in the same way as in the definition of $G_{\phi, \psi}$ as follows. The length of \textsc{Round 0} is $\omega$. After \textsc{Round $n$}, we define an ordinal $\eta_n$ from the players' moves, as in the definition of $G_{\phi, \psi}$. If $\eta_{n+1}\leq\eta_n$ then the game stops and Player I wins immediatelely. Otherwise, the game proceeds to the next round, and the length of \textsc{Round $n+1$} set to be $\omega\cdot\eta_n$. Suppose that $\omega$ many rounds have been played. Then we stop the game and decide a winner of the game as follows. Let $\lambda = \sup\{\eta_n\mid n<\omega\}$ and let $g\colon\lambda\to\power_{\omega_1}(\mathbb{R})$ be defined by
        \[
        g(\eta) = \{x_\xi\mid\xi<\omega\cdot(1+\eta)\}.
        \]
        Also, let $y=\langle m_i\mid i<\omega\rangle$ and let $T$ be the theory that consists of the $\mathcal{L}^+$-sentences $\phi_i$ such that $v_i = 1$. Then Player II wins if and only if
        \begin{itemize}
            \item either $y\notin\mathrm{WO}$, or
            \item $T$ is a theory of an $\omega$-model\footnote{Recall that a model $M$ of some theory is \emph{$\omega$-model} if $\omega^M = \omega$. So, such an $M$ might be ill-founded.} $M$ such that
            \begin{align*}
                M\models\, & \mathsf{ZF}^{-}\footnotemark +V=L(\mathbb{R}_g, \mathcal{F}^g_{<\alpha})\\
                & +\mathsf{AD}+\forall\beta<\omega_1\,(\mathcal{F}^g_{\beta}\text{ is an ultrafilter on }[\power_{\omega_1}(\mathbb{R}_g)]^{\omega^\beta})
            \end{align*}
            and $\|y\|_{\mathrm{WO}}$ is included in the well-founded part of $M$. 
        \end{itemize}
        \footnotetext{$\mathsf{ZF}^-$ denotes $\mathsf{ZF}$ without the Power Set Axiom.}

        \begin{claim}
            Player I cannot have a winning strategy in this game.
        \end{claim}
        \begin{proof}
            This follows from $\mathbf{\Sigma}^1_1$ boundedness as in \cite[Lemma 2.3.13]{TrangThesis}.
        \end{proof}

        By our determinacy assumption, the claim implies that Player II has a winning strategy $\tau$. As in the proofs of the claims in the proof of \cref{DetSol_from_LongGame}, we take a continuous $\in$-chain $\langle H_\xi\mid\xi<\omega_1\rangle$ of countable elementary submodels of $V_\theta$, where $\theta$ is a large enough regular cardinal, such that $\tau\in H_0$. For each $\xi<\omega_1$, let $\mathbb{R}_\xi, M_\xi, \pi_\xi$ be as in the proof of \cref{DetSol_from_LongGame}. Also, let $\alpha$ be the least ordinal such that $\alpha = \omega_1\cap H_\alpha$.

        Now consider a run of $G^{\sharp}$ where Player II follows $\tau$ and Player I plays reals so that $\mathbb{R}_\xi = \{x_{\omega\cdot\xi + 2i}\mid i < \omega\}$. Since $\tau \in M_\xi$, Player I's moves $x_{\omega\cdot\xi + 2i}$ are all in $M_\xi$. The game will stop exactly at length $\alpha$ and we have
        \[
        g = \langle\mathbb{R}_\xi\mid\xi<\alpha\rangle.
        \]
        Then
        \[
        N :=L(\mathbb{R}_g, \mu^g_{<\alpha})=\pi_{\alpha}^{-1}(L(\mathbb{R}, \mu_{<\omega_1})).
        \]
        From now on, we work in $M_{\alpha}[G]$, where $G\subset\mathrm{Col}(\omega, \alpha)$ is $M_{\alpha}$-generic such that some $x\in\mathbb{R}^{M_{\alpha}[G]}$ codes $g$ and $\tau$. Fix such an $x$.

        \begin{claim}\label{HP_implies_sharp}
            In $M_{\alpha}[G]$, the following statement implies that a sharp for $N$ exists in the sense of \cref{dfn:sharp}:
            \begin{itemize}
                \item[(\textsf{HP})\footnotemark] For any ordinal $\gamma>\Theta^N$, if $\gamma$ is $x$-admissible then it is a cardinal in $N$.\footnotetext{\textsf{HP} stands for Harrington's Principle.}
            \end{itemize}
        \end{claim}
        \begin{proof}
            This is an argument in \cite{Har78}. Let $\lambda>\Theta^N$ be an $\omega$-closed cardinal and let $X\prec L_{\lambda^{++}}[x]$ be such that $\lambda\subset X, \lvert X\rvert = \lambda$, and $X^{\omega}\subset X$. Then the transitive collapse of $X$ is $L_{\beta}[x]$ for some $x$-admissible ordinal $\beta<\lambda^+$. Let $\pi\colon L_{\beta}[x]\to X$ be the uncollapse map and let $\kappa = \crit(\pi)$. By \textsf{HP}, $\beta$ is a cardinal of $N$. Then $(\kappa^+)^N\leq\beta$, so we can define an $N$-ultrafilter
            \[
            U = \{A\subset\power(\kappa)\cap N\mid\kappa\in\pi(A)\}.
            \]
            Since $X^\omega\subset X$, $U$ is countably complete. Also, since $\kappa\geq\lambda>\Theta^N$, $U$ is $\mathbb{R}^N$-complete. Then we can define an elementary embedding $j\colon N\to N$ with critical point $\kappa$ by forming iterated ultrapowers of $N$ by $U$.
        \end{proof}

        By \cref{HP_implies_sharp}, it is enough to show that \textsf{HP} holds in $M_{\alpha}[G]$, because then a sharp for $N$ exists in $M_{\alpha}$ by generic absoluteness of the sharp and thus a sharp for $L(\mathbb{R}, \mu_{<\omega_1})$ exists in $V$ by elementarity of $\pi_{\alpha}$.

        Suppose toward a contradiction that $\gamma>\Theta^N$ is $x$-admissible, but not a cardinal in $N$. We may assume that $\gamma<\omega_1^{M_{\alpha}[G]}$ by working in a generic extension of $M_{\alpha}[G]$ if necessary. Let $\kappa$ be the largest cardinal of $N$ below $\gamma$. Let $G^{\sharp}\uphar\mathbb{R}_\alpha$ be the game $G^{\sharp}$ except that players are only allowed to play reals in $\mathbb{R}_\alpha$. In what follows, a wellfounded part of an $\omega$-model is always assumed to be transitive.

        \begin{claim}\label{common_part_of_omega_models}
            For any countable $\omega$-models $\langle P_0, E_0\rangle$ and $\langle P_1, E_1\rangle$ of $\mathsf{ZFC}^{-}$ such that for each $i\in 2$,
            \begin{itemize}
                \item $x\in P_i$,
                \item the ordinal height of the wellfounded part of $\langle P_i, E_i\rangle$ is $\gamma$,
                \item $\langle P_i, E_i\rangle\models\tau$ is a winning strategy for Player II in $G^{\sharp}\uphar\mathbb{R}_\alpha$ and $x$ codes $g$ and $\tau$,
            \end{itemize}
            we have $(\power(\kappa)\cap L(\mathbb{R}_g, \mu^g_{<\alpha}))^{\langle P_0, E_0\rangle} = (\power(\kappa)\cap L(\mathbb{R}_g, \mu^g_{<\alpha}))^{\langle P_0, E_0\rangle}$.
        \end{claim}
        \begin{proof}
            Let $\langle P_0, E_0\rangle, \langle P_1, E_1\rangle$ be any countable $\omega$-models satisfying the above assumption. Let $A\subset\kappa$ be in $L(\mathbb{R}_g, \mu^g_{<\alpha})^{\langle P_0, E_0\rangle}$. We want to show that $A\in L(\mathbb{R}_g, \mu^g_{<\alpha})^{\langle P_1, E_1\rangle}$.

            By Friedman's classical result \cite[Theorem 3.2]{Fri73}, the order type of the ordinals in $\langle P_0, E_0\rangle$ is $\gamma+\gamma\mathbb{Q}$. We take nonstandard ordinals $a_0\in P_0$ and $a_1\in P_1$ such that $\langle P_0, E_0\rangle\models A\in L_{a_0}(\mathbb{R}_g, \mu^g_{<\alpha})$ and
            \[
            \{\xi\in P_0\mid \xi\mathbin{E_0}a_0\}\simeq\{\xi\in P_1\mid\xi\mathbin{E_1}a_1\}\simeq\gamma+\gamma\mathbb{Q}.
            \]
            Let $\mathbb{P}_i = \mathrm{Col}(\omega, a_i)^{\langle N_i, E_i\rangle}$ for each $i\in 2$. Note that there is an isomorphism $\iota\colon\mathbb{P}_0\to\mathbb{P}_1$.

            \begin{subclaim}\label{back-and-forth}
                There are $\langle P_i, E_i\rangle$-generic $H_i\subset\mathbb{P}_i$ for each $i\in 2$ such that
                \begin{enumerate}
                    \item $\iota[H_0] = H_1$, and
                    \item $P_0[H_0]\cap P_1 \cap V_{\gamma} = P_0\cap P_1[H_1] \cap V_{\gamma} = P_0\cap P_1 \cap V_{\gamma}$.
                \end{enumerate}
            \end{subclaim}
            \begin{proof}
                The construction of such generic filters is done by a back-and-forth argument as follows: Fix an enumeration $\vec{d} = \langle d_i\mid i<\omega\rangle$ of $P_0$ and an enumeration $\vec{e}=\langle e_i\mid i<\omega\rangle$ of $P_1$. Also, fix a recursive bijection $b\colon\omega\to\omega\times\omega$ and let $b(k)=\langle b_0(k), b_1(k)\rangle$ for any $k<\omega$.

                We will inductively define a decreasing sequence $\langle p_k\mid k<\omega\rangle$ of conditions in $\mathbb{P}_0$ and a decreasing sequence $\langle q_k\mid k<\omega\rangle$ of conditions in $\mathbb{P}_1$. We start with $p_0 = q_0 = \emptyset$. Suppose that $p_k$ and $q_k$ have been defined. Let $p'_k = \iota^{-1}(q_k)$. To define $p_{k+1}$, we split into three cases depending on what the $b_0(k)$-th element of the enumeration $\vec{d}$ of $P_0$ is:
                \begin{itemize}
                    \item \textsc{Case} 1. $d_{b_0(k)}$ is a dense subset of $\mathbb{P}_0$ in $\langle P_0, E_0\rangle$.
                    \item \textsc{Case} 2. $d_{b_0(k)}$ is a $\mathbb{P}_0$-name in $\langle P_0, E_0\rangle$.
                    \item \textsc{Case} 3. $d_{b_0(k)}$ is not a dense subset of $\mathbb{P}_0$ nor a $\mathbb{P}_0$-name.
                \end{itemize}
                In \textsc{Case} 1, pick any $p_{k+1}\leq_{\mathbb{P}} p'_k$ in $d_{b_0(k)}$. In \textsc{Case} 3, let $p_{k+1} = p'_k$. Suppose that \textsc{Case} 2 holds. If $e_{b_1(k)}\in P_0$, then let $p_{k+1} = p'_k$. Otherwise, let $x_k\in P_0$ be minimal with respect to $\vec{d}$ such that $p'_k$ does not decide whether $\check{x}_k\in d_{b_0(k)}$, if such an $x_k$ exists. If $x_k$ is defined, choose $p_{k+1}\leq_{\mathbb{P}_0} p'_k$ so that
                \begin{align*}
                    x_k\in e_{b_1(k)} &\Rightarrow p_{k+1}\Vdash_{\mathbb{P}_0}^{\langle P_0, E_0\rangle}\check{x}_k\notin d_{b_0(k)},\\
                    x_k\notin e_{b_1(k)} &\Rightarrow p_{k+1}\Vdash_{\mathbb{P}_0}^{\langle P_0, E_0\rangle}\check{x}_k\in d_{b_0(k)}.
                \end{align*}
                If $x_k$ is undefined, let $p_{k+1} =p'_k$. This completes the definition of $p_{k+1}$. Let $q'_k = \iota(p_{k+1})$ and we define $q_{k+1}$ from $q'_k$ in the same way as we defined $p_{k+1}$ from $p'_k$.
                Now let $H_0$ (resp.\ $H_1$) be the upward closure of the $p_k$'s (resp.\ $q_k$'s) for now. 

                To see that the filter $H_0$ is $\langle P_0, E_0\rangle$-generic, let $D\in P_0$ be such that $\langle P_0, E_0\rangle\models$``$D$ is a dense subset of $\mathbb{P}_0$.'' Let $k<\omega$ be such that $D = d_{b_0(k)}$. Then \textsc{Case} 1 of the construction guarantees that $p_{k+1}\in d_{b_0(k)}$, so $H_0\cap D\neq\emptyset$. Symmetrically, one can show the $\langle P_1, E_1\rangle$-genericity of $H_1$.

                To see that $\iota[H_0] \subset H_1$, let $p\in H_0$. Let $k<\omega$ be such that $p_{k+1}\leq p$. Then $\iota(p_{k+1}) = q'_k \in H_1$ and thus $\iota(p)\in H_1$. Symmetrically, one can show $\iota^{-1}[H_1]\subset H_0$ as well, so $\iota[H_0]=H_1$.

                Finally, we show that $P_0[H_0]\cap P_1 \cap V_\gamma \subset P_0\cap P_1$. Suppose otherwise. Then for some $\mathbb{P}_0$-name $\tau\in P_0$, $\tau^{H_0}\in (P_1\setminus P_0)\cap V_\gamma$. By choosing an $E_0$-minimal such $\tau$, we may assume that $\tau^{H_0}\subset P_0$. Let $k<\omega$ be such that $\tau = d_{b_0(k)}$ and $\tau^{H_0} = e_{b_1(k)}$. Now we are in \textsc{Case} 2. If $x_k$ is defined, then $p_{k+1}$ forces $x_k$ to be in the symmetric difference of $d_{b_0(k)}^{H_0}$ and $e_{b_1(k)}$. If $x_k$ is undefined, then $p'_k$ decides whether $\check{x}_k\in d_{b_0(k)}=\tau$ and thus $\tau^{H_0}\in P_0$. In either case, we get a contradiction. Symmetrically, we can show that $P_0\cap P_1[H] \cap V_\gamma \subset P_0\cap P_1$. This completes the proof of the subclaim.
            \end{proof}

            Fix $H_0$ and $H_1$ as in \cref{back-and-forth}. Let $h\in\mathbb{R}$ be coding
            \[
            \{\langle i, j\rangle\in\omega\times\omega\mid H_0(i)\mathbin{E_0}H_0(j)\}=\{\langle i, j\rangle\in\omega\times\omega\mid H_1(i)\mathbin{E_1}H_1(j)\}.
            \]
            Note that $h\in P_0[H_0]\cap P_1[H_1]$. By $\Sigma^1_1$-absoluteness, in both $P_0[H_0]$ and $P_1[H_1]$, $\tau$ is a winning strategy for Player II in $G^{\sharp}\uphar\mathbb{R}_\alpha$. So, the response of $\tau$ when Player I plays $g$ and $h$, denoted by $\langle g, h\rangle * \tau$, is in $P_0[H_0]\cap P_1[H_1]$. In $P_0[H_0]$, $h$ codes a wellordering of length $a_0$ and $\langle g, h\rangle * \tau$ codes an $\omega$-model of $V=L(\mathbb{R}_g, \mu^g_{<\alpha})$ with standard part ${>}a_0$. Therefore, in $P_0[H_0]$, $A$ can be defined from $\langle g, h\rangle * \tau$ in a simple way. Then $A\in P_1[H_1]$, because $\tau(\mathbb{R}_\alpha, h)\in P_1[H_1]$. By \cref{back-and-forth} (2), $A\in P_1$ as desired.
        \end{proof}

        Now fix an $\omega$-model $\langle P, E\rangle$ satisfying the assumption in \cref{common_part_of_omega_models}.

        \begin{claim}\label{well-founded_correctness}
            Let $R\subset\kappa\times\kappa$ be a linear order in $L(\mathbb{R}_g, \mu^g_{<\alpha})^{\langle P, E\rangle}$. Then
            \[
            R\text{ is wellorderd}\iff\langle P, E\rangle\models R\text{ is wellordered}.
            \]
        \end{claim}
        \begin{proof}
            The forward direction is trivial. To show the converse, suppose that $R$ is illfounded. By Jensen's model existence theorem, there is an $\omega$-model $\langle P', E'\rangle$ end-extending $L_{\gamma}[x]$ such that
            \[
            \langle P', E'\rangle\models R\text{ is ill-founded.}
            \]
            By \cref{common_part_of_omega_models}, the witness for the illfoundedness of $R$ is in $(\power(\kappa)\cap L(\mathbb{R}_g, \mu^g_{<\alpha}))^{\langle P, E\rangle} = (\power(\kappa)\cap L(\mathbb{R}_g, \mu^g_{<\alpha}))^{\langle P', E'\rangle}$, so $R$ is also illfounded in $\langle P, E\rangle$.
        \end{proof}

        Since $\gamma$ is the ordinal height of the wellfounded part of $\langle P, E\rangle$, \cref{well-founded_correctness} implies that $\eta:=(\kappa^+)^{L(\mathbb{R}_g, \mu^g_{<\alpha})^{\langle P, E\rangle}} < \gamma$. Now recall that $\kappa$ was defined as the largest cardinal of $N$ below $\gamma$. Then by Jensen's model existence theorem, there is another $\omega$-model $\langle P', E'\rangle$ extending $L_{\gamma}[x]$ such that
        \[
        \langle P', E'\rangle \models \eta\text{ is not a cardinal in }L(\mathbb{R}_g, \mu^g_{<\alpha}).
        \]
        The bijection between $\kappa$ and $\eta$ can be canonically coded into a subset of $\kappa$ in $(\power(\kappa)\cap L(\mathbb{R}_g, \mu^g_{<\alpha}))^{\langle P, E\rangle} = (\power(\kappa)\cap L(\mathbb{R}_g, \mu^g_{<\alpha}))^{\langle P', E'\rangle}$, so $\eta$ is also not a cardinal in $L(\mathbb{R}_g, \mu^g_{<\alpha})^{\langle P, E\rangle}$ by \cref{common_part_of_omega_models}. This contradicts the definition of $\eta$ and finishes the proof of \cref{SharpSol_from_LongGame}.
    \end{proof}

    \subsection{Brief remark on $\mathsf{AD}^+$ and Mouse Capturing}\label{subsection:MCSol}

    \begin{dfn}
        Mouse Capturing ($\mathsf{MC}$) is the statement that for any reals $x, y\in\mathbb{R}$, if $x$ is ordinal definable from $y$, then there is an $\omega_1+1$-iterable sound $y$-premouse $\mathcal{M}$ projecting to $\omega$ such that $x\in\mathcal{M}$.
    \end{dfn}

    The core model induction arguments in \cite{DetLRmu} can be generalized to show:

    \begin{thm}\label{mouse_capturing}
        Suppose that
        \[
        L(\mathbb{R}, \mu_{<\omega_1})\models\mathsf{AD}+\forall\alpha<\omega_1\,(\mu_{\alpha}\text{ is an ultrafilter on }[\power_{\omega_1}(\mathbb{R})]^{\omega^\alpha}).
        \]
        Then the following hold in $L(\mathbb{R}, \mu_{<\omega_1})$:
        \begin{enumerate}
            \item $\mathrm{Lp}(\mathbb{R})\models\mathsf{AD}^+ + \Theta = \theta_0 + \mathsf{MC}$,
            \item $\power(\mathbb{R}) = \mathrm{Lp}(\mathbb{R}) \cap \power(\mathbb{R})$, and thus
            \item $\mathsf{AD}^+ + \Theta = \theta_0$.
        \end{enumerate}
    \end{thm}

    Although the conclusion of \cref{mouse_capturing} is necessary in the proof of \cref{reversed_dervied_model_theorem} and \cref{capturing_OD_sets}, we do not have to show \cref{mouse_capturing} for \cref{mainthm}, as the failure of $\mathsf{MC}$ would have higher consistency strength than the large cardinal assumption of \cref{mainthm} (cf.\ \cite{Woodin_mouse_sets,St08}). Therefore, we omit the proof of \cref{mouse_capturing}.

    \section{Structure theory of the Solovay-type model}

    We summarize Trang--Woodin's work on the model $L(\mathbb{R}, \mu_{<\omega_1})$ in \cite{TrangThesis}. Although this section does not contain any new result, the stationary-tower-free proof of the derived model theorem for $L(\mathbb{R}, \mu_{<\omega_1})$ is due to us. We use this proof in the next section.

    \subsection{Derived model theorem}

    Recall the notation introduced in the beginning of \cref{DetSol_from_LG}.

    \begin{dfn}\label{derived_Solovay_model}
        Let $\lambda$ be a limit of Woodin cardinals such that the order type of all Woodin cardinals below $\lambda$ is $\lambda$. Let $G\subset\mathrm{Col}(\omega, {<}\lambda)$ be $V$-generic. In $V[G]$, we define $g\colon\lambda\to\power_{\omega_1}(\mathbb{R})$ by
        \[
        g(\eta) = \bigcup_{\xi < \eta}\mathbb{R}^{V[G\uphar\xi]}
        \]
        for all $\eta < \lambda$, where $G\uphar\xi = G\cap\mathrm{Col}(\omega, \xi)$. The \emph{derived $\Delta$-Solovay model at $\lambda$ in $V[G]$} is defined as
        \[
        L(\mathbb{R}^*_G, \mu_{<\lambda}^G) := L(\mathbb{R}_g, \mu_{<\lambda}^g).
        \]
        Also, for each $\alpha < \lambda$, we write $\mathcal{F}^G_\alpha = \mathcal{F}^g_\alpha$ and  $\mu^G_\alpha = \mu^g_\alpha$. Note that $\lambda = \omega_1^{V[G]}$.
    \end{dfn}

    The following theorem was claimed in \cite{TrangThesis}, but the proof below is due to us.

    \begin{thm}[Trang--Woodin \cite{TrangThesis}]\label{derived_model_theorem}
        Suppose that there is a limit of Woodin cardinals $\lambda$ such that the order type of Woodin cardinals below $\lambda$ is $\lambda$. Let $G\subset\mathrm{Col}(\omega, {<}\lambda)$ be $V$-generic. Then in $V[G]$, the following hold:
        \begin{enumerate}
            \item Let $x\in\mathbb{R}^*_G$ and let $\eta<\lambda$ be such that $x\in\mathbb{R}^{V[G\uphar\eta]}$. If
            \[
            \exists B\in\power(\mathbb{R}^*_G)\cap L(\mathbb{R}^*_G, \mu^G_{<\omega_1})\,(\mathrm{HC}^*_G, \in, B)\models\phi[x],
            \]
            then
            \[
            \exists B\in\mathrm{Hom}^{V[G\uphar\eta]}_{{<}\lambda}\,(\mathrm{HC}^{V[G\uphar\eta]}, \in, B)\models\phi[x].
            \]
            \item $L(\mathbb{R}^*_G, \mu^G_{<\lambda})\models\mathsf{AD}^+ \land \forall\alpha<\omega_1\,(\mu^G_{\alpha}\text{ is an ultrafilter on }[\power_{\omega_1}(\mathbb{R}^*_G)]^{\omega^\alpha})$.
        \end{enumerate}
    \end{thm}
    \begin{proof}
        By \cite[Lemma 2.2.6]{TrangThesis}, we may assume that
        \begin{equation}\label{consequence_of_Radin_forcing}
            L(\mathbb{R}, \mu_{<\omega_1})\models\forall\alpha<\omega_1\,(\mu_\alpha\text{ is an ultrafilter on }[\power_{\omega_1}(\mathbb{R})]^{\omega^\alpha}).
        \end{equation}
        Let $\pi\colon P\to V_{\theta}$ be elementary, where $P$ is countable and transitive, $\theta$ is sufficiently large, and $\lambda\in\ran(\pi)$. Let $\overline{\lambda}=\pi^{-1}(\lambda)$. Then $P$ is sufficiently iterable for Neeman's genericity iteration. Fix a continuous $\in$-chain $\langle H_{\xi}\mid\xi<\omega_1\rangle$ of countable elementary submodels of $V_{\theta}$ such that $P\in H_0$. For each $\xi<\omega_1$, let
        \[
        \mathbb{R}_\xi = \mathbb{R} \cap H_\xi,
        \]
        let $M_{\xi}$ be the transitive collapse of $H_{\xi}$, and let $\pi_{\xi}\colon M_{\xi}\to V_{\theta}$ be the uncollapse map. Now, we form a sequence of genericity iterates $\langle P_{\xi}\mid\xi\leq\alpha\rangle$ as follows:
        \begin{itemize}
            \item $P_0 = P$.
            \item If $\xi<\alpha$, then let $P_{\xi+1}$ be an $\mathbb{R}_{-1+\xi+1}$-genericity iterate of $P_{\xi}$ using the image of the Woodin cardinals $\delta_{\xi+n}$ for $n<\omega$ so that every initial segment of this iteration belongs to $M_\xi$. Note that the entire iteration only belongs to $M_{\xi +1}$.
            \item If $\xi\leq\alpha$ is limit, then let $P_{\xi}$ be the direct limit of $\langle P_{\eta}\mid\eta<\xi\rangle$ along the iteration maps.
            \item If $\xi$ is the image of $\overline{\lambda}$ under the iteration embedding from $P_0$ to $P_\xi$, we stop the construction and let $\alpha = \xi$.
        \end{itemize}
        Each $P_\xi$ can be realized back to $V_\theta$ and thus it is wellfounded. As usual, we arrange that there is a $P_{\alpha}$-generic $H\subset\mathrm{Col}(\omega, {<}\alpha)$ such that 
        \[
        \langle\mathbb{R}^*_{H\uphar\delta_{\xi}}\mid\xi<\alpha\rangle = \langle\mathbb{R}_\xi\mid\xi<\alpha\rangle
        \]
        and thus $\mathbb{R}^*_H = \mathbb{R}_\alpha$.

        \begin{claim}\label{DSM_is_init_seg_of_SM}
            The derived $\Delta$-Solovay model of $P_{\alpha}$ at $\alpha$ in $P_{\alpha}[h]$ is an $L$-initial segment of the $\Delta$-Solovay model in $M_{\alpha}$, i.e.,
            \[
            L(\mathbb{R}^*_H, \mu_{<\alpha}^H)^{P_\alpha[H]} = L_{P_\alpha\cap\mathrm{Ord}}(\mathbb{R}, \mu_{<\omega_1})^{M_\alpha}.
            \]
        \end{claim}
        \begin{proof}
            Note that $\alpha = \omega_1^{M_\alpha}$. Let $\langle\overline{\mathcal{F}}_\beta\mid\beta<\alpha\rangle= \pi^{-1}_\alpha(\langle\mathcal{F}_\beta\mid\beta<\omega_1\rangle)$ and we write $\overline{\mu}_\beta = \overline{\mathcal{F}}_\beta\cap L(\mathbb{R}, \mu_{<\omega_1})^{M_\alpha}$ for $\beta<\alpha$. Using this notation,
            \[
            L(\mathbb{R}, \mu_{<\omega_1})^{M_\alpha} = L_{M_\alpha\cap\ord}(\mathbb{R}_\alpha, \overline{\mu}_{<\alpha}).
            \]
            
            Now we claim that for all $\beta < \alpha$, $\overline{\mathcal{F}}_{\beta}\subset\mathcal{F}^H_{\beta}$.
            To see this, fix $\beta<\alpha$ and $Z\in\overline{\mathcal{F}}_\beta$ witnessed by $F\in M_\alpha$. Then $\mathbb{R}^*_{H\upharpoonright\delta_\xi} = \mathbb{R}_\xi$ is closed under $F$ for all sufficiently large $\xi<\alpha$, so $Z\in\mathcal{F}^H_\beta$.

            By (\ref{consequence_of_Radin_forcing}), $\overline{\mu}_{\beta}$ is an ultrafilter in $L_{M_\alpha\cap\ord}(\mathbb{R}_\alpha, \overline{\mu}_{<\alpha})$. Then by induction on $\gamma\in P_\alpha\cap\ord$, we have
            \[
            \mathcal{F}^H_\beta\cap L_{\gamma}(\mathbb{R}^*_H, \mu^H_{<\alpha}) = \overline{\mathcal{F}}_\beta\cap L_{\gamma}(\mathbb{R}_\alpha, \overline{\mu}_{<\alpha})
            \]
            for all $\beta < \alpha$ and thus
            \[
            L_{\gamma}(\mathbb{R}^*_H, \mu^H_{<\alpha}) = L_{\gamma}(\mathbb{R}_\alpha, \overline{\mu}_{<\alpha}),
            \]
            which completes the proof of the claim.
        \end{proof}
        
        Now we show (1). For simplicity, we assume $x=\emptyset$ and $\eta=0$. By \cref{DSM_is_init_seg_of_SM}, the assumption of (1) implies that for some $B\in L(\mathbb{R}, \mu_{<\omega_1})$, $(\mathrm{HC}, \in, B)\models\phi$. We call such a set $B$ a \emph{$\phi$-witness}. Let $\beta_0$ be the least ordinal $\beta$ such that there is a $\phi$-witness in $ L_\beta(\mathbb{R}, \mu_{<\omega_1})$. Furthermore, by minimizing ordinal parameters, we can find a formula $\psi(u, v)$ and a real $y\in\mathbb{R}$ such that $(\mathrm{HC}, \in, B_0)\models\phi$, where
        \[
        B_0 = \{u\in\mathbb{R}\mid \langle L_{\beta_0}(\mathbb{R}, \mu_{<\omega_1}), \in, \mu_{<\omega_1}\rangle\models\psi[u, y]\}.
        \]
        Let $\lambda'<\lambda$ be large enough so that any weakly $\lambda'$-homogeneous set of reals is in $\mathrm{Hom}_{<\lambda}$. Also, let $\rho$ be the least Woodin cardinal above $\lambda'$. Let $W$ be the set of (real codes of) $2^\omega$-closed\footnote{An iteration tree $\mathcal{T}$ is \emph{$2^\omega$-closed} if for all $\alpha+1<\lh(\mathcal{T})$,  $M^{\mathcal{T}}_\alpha\models$``$\ult(V, E^{\mathcal{T}}_\alpha)$ is closed under $2^\omega$-sequences.''} iteration trees $\mathcal{T}$ on $P$ of length $\omega+1$ above $\lambda'$\footnote{An iteration tree $\mathcal{T}$ is \emph{above $\lambda'$} if $\crit(E^\mathcal{T}_\alpha)>\lambda'$ for all $\alpha+1<\lh(\mathcal{T})$.} such that $\pi\mathcal{T}$ is a (wellfounded) iteration tree. Windszus' theorem (see \cite[Lemma 1.1]{St07}) asserts that $W$ is $\pi(\lambda')$-homogeneously Suslin. Then the following claim implies that $B$ is $\mathrm{Hom}_{{<}\lambda}$, which completes the proof of (1).
    
        \begin{claim}
            For any $u\in\mathbb{R}$, the following are equivalent:
            \begin{itemize}
                \item $u\in B_0$.
                \item There are an iteration tree $\mathcal{T}$ on $P$ in $W$ and $g\subset\mathrm{Col}(\omega, i^{\mathcal{T}}_{0, \omega}(\rho))$ generic over $M^{\mathcal{T}}_{\omega}$ such that $u, y\in M^{\mathcal{T}}_{\omega}[g]$, some initial segment of the derived $\Delta$-Solovay model of $M^{\mathcal{T}}_{\omega}[g]$ at $i^{\mathcal{T}}_{0, \omega}(\overline{\lambda})$ has a $\phi$-witness, and the least such initial segment satisfies $\psi[u, y]$.
            \end{itemize}
        \end{claim}
        \begin{proof}
            This is proved in the same way as \cite[Claim 2]{St07}.
            Let $u\in\mathbb{R}$. By Neeman's genericity iteration, we can take $\mathcal{T}\in W$ and $g\subset\mathrm{Col}(\omega, i^{\mathcal{T}}_{0, \omega}(\rho))$ such that $u, y\in M^{\mathcal{T}}_{\omega}[g]$. As in the proof of the last claim, we iterate $M^{\mathcal{T}}_{\omega}$ to $N$ so that the derived $\Delta$-Solovay model of $N$ is an initial segment of the $\Delta$-Solovay model of some elementary submodel of $V$. Also, there is a realization map from $N$ to $V_\theta$. By elementarity of this realization map, the assumption of (1) implies that the derived $\Delta$-Solovay model of $N$ has a $\phi$-witness. Then $u\in B_0$ if and only if the least initial segment of the derived $\Delta$-Solovay model of $N$ containing a $\phi$-witness satisfies $\psi[u, y]$. By elementarity of the iteration embedding from $M_\omega^\mathcal{T}$ to $N$, the claim follows.
        \end{proof}
    
        (2) follows from (1) and its proof. For the proof of $\mathsf{AD}^+$, see the discussion after Lemma 6.4 in \cite{St09} or \cite[Lemma 2.4]{Tr15}. This completes the proof of \cref{derived_model_theorem}.
    \end{proof}

    \subsection{Prikry-like forcings}\label{subsection:Prikry-like_forcing}

    Throughout this subsection, we assume that
    \begin{equation}\tag{DetSol}\label{global_assump_for_Prikry_forcing}
        L(\mathbb{R}, \mu_{<\omega_1})\models\mathsf{AD} + \forall\alpha<\omega_1\,(\mu_{\alpha}\text{ is an ultrafilter on }[\power_{\omega_1}(\mathbb{R})]^{\omega^\alpha})
    \end{equation}
    and work in $L(\mathbb{R}, \mu_{<\omega_1})$. Note that by \cref{mouse_capturing}, (\ref{global_assump_for_Prikry_forcing}) implies $\mathsf{AD}^+ + \mathsf{MC}$ in $L(\mathbb{R}, \mu_{<\omega_1})$. We introduce a Prikry-like forcing $\mathbb{P}_{\Delta}$, which is a slight modification of the poset $\mathbb{P}_{\omega_1}$ in \cite{TrangThesis}. Such a modification seems necessary to guarantee the Mathias property for $\mathbb{P}_{\Delta}$. However, all the results in this subsection should be attributed to Trang and Woodin.

    For any real $x\in\mathbb{R}$, we let
    \[
    d_x = \{y\in\mathbb{R}\mid\mathrm{HOD}_x = \mathrm{HOD}_y\}.
    \]
    This is called the \emph{$\Sigma^2_1$-degree} of $x$. Let $\mathbb{D}$ be the set of all $\Sigma^2_1$-degrees. The $\Sigma^2_1$ degrees are ordered by $d_x\leq d_y \iff x\in \mathrm{HOD}_y$. For each $\alpha<\omega_1$, let $\mathbb{D}_{\alpha}$ be the set of increasing $\omega^{1+\alpha}$-sequences of $\Sigma^2_1$-degrees. The Prikry-like forcing $\mathbb{P}_{\Delta}$ will add an increasing $\omega_1^V$-sequence of $\Sigma^2_1$-degrees. We will see in \cref{mouse_from_DetSol} that this yields a premouse with a limit of Woodin cardinals $\lambda$ such that the order type of Woodin cardinals below $\lambda$ is $\lambda$. To define $\mathbb{P}_{\Delta}$, we need a sequence $\langle\nu_{\alpha}\mid\alpha<\omega_1\rangle$ of ultrafilters on $\mathbb{D}_{\alpha}$. To define these ultrafilters, we simultaneously define Prikry-like forcings $\mathbb{P}_\alpha$ associated to $\langle\nu_\beta\mid\beta<\alpha\rangle$ by induction. The definitions below are slight modifications of definitions in \cite{TrangThesis}.

    First, note that $\mathsf{AD}$ implies that the cone filter $\nu$ on $\mathbb{D}$ defined by
    \[
    A\in\nu \iff \exists d\in\mathbb{D}\,\forall e\geq d\,(e\in A)
    \]
    is an ultrafilter. We define the (tree) Prikry-like forcing $\mathbb{P}_{\Sigma^2_1}$ associated to $\nu$ as in \cite[Section 4]{Ket11}: Let $T$ be a tree on $\omega\times\ord$ projecting to a universal $\Sigma^2_1$ set. The conditions of $\mathbb{P}_{\Sigma^2_1}$ are the pairs $\langle p, U\rangle$ such that
    \begin{itemize}
        \item $p=\langle d_i\mid i<n\rangle\in\mathbb{D}^{<\omega}$ such that for all $i$ with $i+1<n$, $d_i\in L[T, d_{i+1}(0)]$ and $d_i$ is countable in $L[T, d_{i+1}(0)]$, and
        \item $U\colon\mathbb{D}^{<\omega}\to\nu$.
    \end{itemize}
    The order of $\mathbb{P}_{\Sigma^2_1}$ is defined by $\langle q, W\rangle\leq\langle p, U\rangle$ iff
    \begin{enumerate}
        \item $q\supseteq p$,
        \item for all $i\in\dom(q)\setminus\dom(p), q(i)\in U(p\uphar i)$,
        \item $W(r)\subset U(r)$ for all $r\in\mathbb{D}^{<\omega}$.
    \end{enumerate}
    For any generic $G\subset\mathbb{P}_{\Sigma^2_1}$, we write
    \[
    \vec{d}(G) = \bigcup\{p\mid \exists U\,(\langle p, U\rangle\in G)\}.
    \]
    We define an ultrafilter $\nu_0$ on $\mathbb{D}_0$ as follows. For all $A\subset\mathbb{D}_0$, we set $A\in\nu_0$ if and only if for any $\infty$-Borel code $S$ for $A$, for $\mu_0$-almost all $\sigma\in\power_{\omega_1}(\mathbb{R})$,
    \[
    L(\sigma)[T, S]\models\text{``}\exists\langle\emptyset, U\rangle\in\mathbb{P}_{\Sigma^2_1}\,(\langle\emptyset, U\rangle\Vdash_{\mathbb{P}_{\Sigma^2_1}}\vec{d}(\dot{G})\in A_S),\text{''}
    \]
    where $\dot{G}$ is the canonical name for a generic filter on $\mathbb{P}_{\Sigma^2_1}$, and $A_S\subset\mathbb{D}_0$ is the generic interpretation of $S$.\footnote{See e.g.\ \cite{Ket11} for basic facts on the notion of $\infty$-Borel codes.} To check that $\nu_0$ is well-defined, one needs to show that for $\mu_0$-almost all $\sigma$, $L(\sigma)[T, S]\models$``$\mathsf{AD}^+\land\mathbb{R} = \sigma$'' and that whether $A\in\nu_0$ does not depend on the choice of $\infty$-Borel code for $A$. See \cite{TrangThesis} or \cite{Tr15} for the proof.
    
    Assume now inductively that $\nu_{\alpha}$ has been defined. Then we define $\mathbb{P}_{\alpha}$ as the Prikry-like forcing associated with $\nu_\alpha$, in the same way that we defined $\mathbb{P}_{\Sigma^2_1}$ from $\nu$. Moreover, for any $g = \langle\sigma_\xi\mid\xi < \omega^{\alpha+1}\rangle \in [\power_{\omega_1}(\mathbb{R})]^{\omega^{\alpha+1}}$, let $\mathbb{R}_g = \ran(g)$ and let $\mathcal{F}^g_\alpha$ be the set of all $Z\subset[\power_{\omega_1}(\mathbb{R}_g)]^{\omega^\alpha}$ such that for all but finitely many $m < \omega$,
    \[
    \langle g(\omega^\alpha\cdot m + \xi)\mid\xi < \omega^\alpha\rangle\in Z.
    \]
    Using this notation, we define an ultrafilter $\nu_{\alpha+1}$ on $\mathbb{D}_{\alpha+1}$ by letting $A\in\nu_{\alpha+1}$ if and only if for any $\infty$-Borel code $S$ for $A$, for $\mu_{\alpha+1}$-almost all $g\in[\power_{\omega_1}(\mathbb{R})]^{\omega^{\alpha+1}}$,
    \[
    \mathrm{HOD}_{\mathbb{R}_g\cup\{\mathcal{F}_{\alpha}^g\}}\models\text{``}\exists\langle\emptyset, U\rangle\in\mathbb{P}_{\alpha}^g\,(\langle\emptyset, U\rangle\Vdash_{\mathbb{P}_{\alpha}^g}\vec{d}(\dot{G})\in A_S),\text{''}
    \]
    where $\mathbb{P}_{\alpha}^g$ is the Prikry-like forcing associated with (the restriction of) $\mathcal{F}_\alpha^g$ defined in $\mathrm{HOD}_{\mathbb{R}_g\cup\{\mathcal{F}_{\alpha}^g\}}$ in the same way that $\mathbb{P}_\alpha$ is associated with $\mu_\alpha$. To see the well-definedness of $\nu_{\alpha+1}$, one needs to show that for $\mu_{\alpha+1}$-almost all $g$, $\mathrm{HOD}_{\mathbb{R}_g\cup\{\mathcal{F}_{\alpha}^g\}}$ satisfies $\mathsf{AD}^+, \mathbb{R} = \mathbb{R}_g$, and  that (the restriction of) $\mathcal{F}_\alpha^g$ is an ultrafilter. See \cite{TrangThesis} for the proof of the case where $\alpha = 0$. The proof of the general case is similar.

    For the limit step of the induction, assume that $\alpha$ is limit and $\langle\nu_{\beta}\mid\beta<\alpha\rangle$ has been defined. To define $\mathbb{P}_\alpha$ and $\nu_\alpha$, we temporarily fix cofinal increasing functions $f_\beta\colon\omega\to\beta$ for $\beta\leq\alpha$. We say that a finite sequence $\langle d_i\mid i<n\rangle$ is called a \emph{$\mathbb{D}$-sequence relative to $f_\beta$} if $d_i\in\mathbb{D}_{f_\alpha(i)}$ for all $i<n$. The conditions of $\mathbb{P}_{\alpha}$ are the pairs $\langle p, U\rangle$ such that 
    \begin{itemize}
        \item $p=\langle d_i\mid i<n\rangle$ is a $\mathbb{D}$-sequence relative to $f_\alpha$ such that for all $i$ with $i+1<n$, $d_i \in L[T, d_{i+1}(0)]$ and $d_i$ is countable in $L[T, d_{i+1}(0)]$.
        \item $U$ is a function defined on all $\mathbb{D}$-sequences relative to $f_\alpha$ such that for all $p\in\dom(U), U(p)\in\nu_{f_\alpha(\lvert p\rvert)}$.
    \end{itemize}
    The order of $\mathbb{P}_\alpha$ is defined by $\langle q, W\rangle\leq\langle p, U\rangle$ iff
    \begin{enumerate}
        \item $q\supseteq p$,
        \item for all $i\in\dom(q)\setminus\dom(p), q(i)\in U(p\uphar i)$,
        \item $W(r)\subset U(r)$ for all $r\in\dom(U)=\dom(W)$.
    \end{enumerate}
    For any $g = \langle\sigma_\beta\mid\beta < \omega^\alpha\rangle \in [\power_{\omega_1}(\mathbb{R})]^{\omega^\alpha}$, let $\mathbb{R}_g = \ran(g)$ and for $\beta<\alpha$, let $\mathcal{F}^g_\beta$ be the set of all $Z\subset[\power_{\omega_1}(\mathbb{R}_g)]^{\omega^\beta}$ such that for all but finitely many $m < \omega$,
    \[
    \langle g(f_\beta(m)+\xi))\mid\xi < \omega^\beta\rangle \in Z.
    \]
    We define an ultrafilter $\nu_\alpha$ on $\mathbb{D}_\alpha$ by letting $A\in\nu_{\alpha}$ if and only if for any $\infty$-Borel code $S$ for $A$, for $\mu_{\alpha}$-almost all $g\in[\power_{\omega_1}(\mathbb{R})]^{\omega^\alpha}$,
    \[
    \mathrm{HOD}_{\mathbb{R}_g\cup\{\mathcal{F}_{\beta}^g\mid\beta<\alpha\}}\models\text{``}\mathsf{AD}^+ + \mathbb{R}_g = \mathbb{R} + \exists\langle\emptyset, U\rangle\in\mathbb{P}_\alpha^g\,(\langle\emptyset, U\rangle\Vdash_{\mathbb{P}_{\alpha}^g}\vec{d}(\dot{G})\in A_S),\text{''}
    \]
    where $\mathbb{P}_\alpha^g$ is the Prikry-like forcing associated with (the restrictions of) $\langle\mathcal{F}^g_\beta\mid\beta<\alpha\rangle$ defined in $\mathrm{HOD}_{\mathbb{R}_g\cup\{\mathcal{F}_{\beta}^g\mid\beta<\alpha\}}$ in the same way that $\mathbb{P}_\alpha$ is associated with $\langle\nu_\beta\mid\beta<\alpha\rangle$. One can show that $\nu_\alpha$ is independent of the choice of cofinal functions $f_\beta\colon\omega\to\beta$, where $\beta\leq\alpha$. See the proof of \cite[Theorem 1.2(2)]{Tr15}.

    We are finally ready to introduce the main Prikry-like forcing $\mathbb{P}_\Delta$.

    \begin{dfn}
        A finite sequence $\langle d_i\mid i<n\rangle$ is called a \emph{$\mathbb{D}$-sequence} if there is an increasing sequence $\langle\alpha_i\mid i<n\rangle$ of countable ordinals such that $d_i\in\mathbb{D}_{\alpha_i}$ for all $i<n$. The conditions of $\mathbb{P}_{\Delta}$ are the pairs $\langle p, U\rangle$ such that 
    \begin{itemize}
        \item $p=\langle d_i\mid i<n\rangle$ is a $\mathbb{D}$-sequence such that for all $i$ with $i+1<n$, $d_i$  is in $L[T, d_{i+1}(0)]$ and is countable there.
        \item $U$ is a function defined on all $\mathbb{D}$-sequences such that
        \begin{enumerate}
            \item for all $p\in\dom(U), U(p)\subset\bigcup_{\alpha<\omega_1}\mathbb{D}_{\alpha}$, and
            \item for club many $\alpha<\omega_1$, $U(p)\cap\mathbb{D}_{\alpha}\in\nu_{\alpha}$.
        \end{enumerate}
    \end{itemize}
    The order of $\mathbb{P}_{\Delta}$ is defined by $\langle q, W\rangle\leq\langle p, U\rangle$ iff
    \begin{enumerate}
        \item $q\supseteq p$,
        \item for all $i\in\dom(q)\setminus\dom(p)$, $q(i)\in U(p\upharpoonright i)$,
        \item $W(r)\subset U(r)$ for all $\mathbb{D}$-sequences $r$.
    \end{enumerate}
    For any generic $G\subset\mathbb{P}_{\Delta}$, we define
    \[
    \vec{d}(G) = \bigcup\{p\mid \exists U\,(\langle p, U\rangle\in G)\}.
    \]
    \end{dfn}

    Recall that for any ultrafilter $\mathcal{U}$ on $\kappa$ and a sequence $\langle\mathcal{V}_\alpha\mid\alpha<\kappa\rangle$, where each $\mathcal{V}_\alpha$ is an ultrafilter on some set $X_\alpha$, the Fubini sum of $\langle\mathcal{V}_\alpha\mid\alpha<\kappa\rangle$ over $\mathcal{U}$, denoted by $\sum_\mathcal{U}\mathcal{V}_\alpha$, is an ultrafilter on $\bigcup_{\alpha<\kappa}(\{\alpha\}\times X_\alpha)$ defined by
    \[
    X\in\sum_{\mathcal{U}}\mathcal{V}_\alpha \iff \{\alpha<\kappa\mid\{A\in X_\alpha\mid\langle\alpha, A\rangle\in X\}\in\mathcal{V}_\alpha\}\in\mathcal{U}.
    \]
    One can see $\mathbb{P}_\Delta$ as the Prikry forcing associated to the ultrafilter $\Sigma_{\mathcal{C}}\,\nu_{\alpha}$ on $\bigcup_{\alpha<\omega_1}\mathbb{D}_\alpha\simeq\bigcup_{\alpha<\omega_1}(\{\alpha\}\times\mathbb{D}_\alpha)$, where $\mathcal{C}$ is the club filter on $\omega_1$. Note that by our assumption (\ref{global_assump_for_Prikry_forcing}), $\mathcal{C}$ and the $\nu_\alpha$'s are all ultrafilters, and thus $\Sigma_\mathcal{C}\,\nu_\alpha$ is also an ultrafilter.

    A simple genericity argument shows that, for any generic $G\subset\mathbb{P}_\Delta$, if $\vec{d}(G)(i)\in\mathbb{D}_{\alpha_i}$ for all $i<\omega$, then $\sup_{i<\omega}\alpha_i = \omega_1^V$. One can find the proof of the following lemma in \cite[Theorem 4.1]{Ket11}. The essentially same argument can be found in \cite[Lemma 6.18]{KW10} and \cite[Claim 6.38]{StW16}, too.

    \begin{lem}
        Under the assumption (\ref{global_assump_for_Prikry_forcing}), the following holds in $L(\mathbb{R}, \mu_{<\omega_1})$: Let $\langle p, U\rangle\in\mathbb{P}_{\Delta}$ and let $\phi$ be a sentence of the forcing language. Then there is a $\langle p, W\rangle\leq\langle p, U\rangle$ such that
        \begin{enumerate}
            \item $\langle p, W\rangle$ decides $\phi$, and
            \item $\langle p, W\rangle$ is ordinal definable from $\langle p, U\rangle$ and $\phi$.
        \end{enumerate}
    \end{lem}

    Let $U$ be such that $\langle\emptyset, U\rangle\in\mathbb{P}_{\Delta}$. We say that $U$ is \emph{uniform}\footnote{This terminology comes from the proof of Lemma 6.40 in \cite{StW16}.} if whenever $q$ is a subsequence of $p$, $U(p)\subset U(q)$. Note that for any $\langle\emptyset, U\rangle$, one can find a uniform $W$ such that $\langle\emptyset, W\rangle\leq\langle\emptyset, U\rangle$ in a definable way. The following lemma is proved in \cite[Theorem 4.2]{Ket11} and essentially in \cite[Lemma 6.40]{StW16}, too.

    \begin{lem}[Mathias property]\label{Mathias_condition}
        Assume (\ref{global_assump_for_Prikry_forcing}). If a sequence $\vec{d}=\langle d_i\mid i<\omega\rangle$ satisfies that for all uniform $U$, $\exists i\,\forall j\geq i\, (d_j\in U(\vec{d}\uphar j))$, then $\vec{d}$ is $\mathbb{P}_\Delta$-generic over $L(\mathbb{R}, \mu_{<\omega_1})$.
    \end{lem}

    \begin{cor}
        Assume (\ref{global_assump_for_Prikry_forcing}). Then any infinite subsequence of a $\mathbb{P}_{\Delta}$-generic sequence over $L(\mathbb{R}, \mu_{<\omega_1})$ is also $\mathbb{P}_{\Delta}$-generic over $L(\mathbb{R}, \mu_{<\omega_1})$.
    \end{cor}
    \begin{proof}
        Let $\vec{d}$ be $\mathbb{P}_\Delta$-generic over $L(\mathbb{R}, \mu_{<\omega_1})$ and let $\vec{e}=\langle e_i\mid i<\omega\rangle=\langle d_{k_i}\mid i<\omega\rangle$ be an infinite subsequence of $\vec{d}$. By genericity, for any uniform $U$, there is an $i_0<\omega$ such that $\forall j\geq i_0\,(d_j\in U(\vec{d}\uphar j))$. Then for all $j\geq i_0$,
        \[
        e_j = d_{k_j}\in U(\vec{d}\uphar k_j)\subset U(\vec{e}\uphar j),
        \]
        as $\vec{e}\uphar j$ is a subsequence of $\vec{d}\uphar k_j$ and $U$ is uniform. By \cref{Mathias_condition}, $\vec{e}$ is $\mathbb{P}_{\Delta}$-generic.
    \end{proof}

    One can adopt the argument in \cite{TrangThesis,Tr15} for $L(\mathbb{R}, \mu)$ to show the following:

    \begin{thm}[Trang--Woodin, \cite{TrangThesis}]\label{reversed_dervied_model_theorem}
        Assume (\ref{global_assump_for_Prikry_forcing}) and let $G\subset\mathbb{P}_{\Delta}$ be generic over $L(\mathbb{R}, \mu_{<\omega_1})$. Then in $L(\mathbb{R}, \mu_{<\omega_1})[G]$, there is a proper class model $N$ of $\mathsf{ZFC}$ such that letting $\lambda$ be the supremum of all Woodin cardinals of $N$,
        \begin{itemize}
            \item $\lambda$ is the order type of all Woodin cardinals of $N$, and
            \item there is an $N$-generic $H\subset\mathrm{Col}(\omega, {<}\lambda)$ such that $L(\mathbb{R}, \mu_{<\omega_1})$ is the derived $\Delta$-Solovay model of $N$ at $\lambda$ via $H$.
        \end{itemize}
    \end{thm}

    The important corollary of this is the following form of $\Sigma_1$ reflection in the $\Delta$-Solovay model, which will be used in the next section.

    \begin{thm}[Trang--Woodin, \cite{TrangThesis}]\label{Sigma_1_reflection}
        Assume (\ref{global_assump_for_Prikry_forcing}). Then
        \[
        \langle L_{\delta^2_1}(\mathbb{R}, \mu_{<\omega_1}), \in, \mu_{<\omega_1}\cap L_{\delta^2_1}(\mathbb{R}, \mu_{<\omega_1})\rangle \prec_{\Sigma_1} \langle L(\mathbb{R}, \mu_{<\omega_1}), \in, \mu_{<\omega_1}\rangle.
        \]
        Here, we consider $\Sigma_1$-elementarity in the language of set theory with an additional predicate for $\mu_{<\omega_1}$.
    \end{thm}
    \begin{proof}
        By the proof of \cite[Theorem 1.2]{Tr15}, this follows from \cref{derived_model_theorem} and \cref{reversed_dervied_model_theorem}.
    \end{proof}

    \section{A mouse from determinacy in the Solovay-type model}\label{mouse_from_DetSol}
    
    We will prove the following theorem in this section.
    
    \begin{thm}\label{mouse_from_determinacy}
        Suppose that
        \[
        L(\mathbb{R}, \mu_{<\omega_1})\models\mathsf{AD} + \forall\alpha<\omega_1(\mu_{\alpha}\text{ is an ultrafilter on }[\power_{\omega_1}(\mathbb{R})]^{\omega^\alpha})
        \]
        and the sharp for $L(\mathbb{R}, \mu_{<\omega_1})$ exists. Then there is an $\omega_1$-iterable active premouse $\mathcal{M}$ with a limit $\lambda$ of Woodin cardinals of $\mathcal{M}$ such that the order type of Woodin cardinals below $\lambda$ of $\mathcal{M}$ is $\lambda$.
    \end{thm}

    \subsection{$\Gamma$-suitability and $A$-iterability}

    We review the definition of $\Gamma$-suitability and $A$-iterability in \cite{scales_at_weak_gap,scales_in_hybrid_mice} with slight adaptation to our context. In particular, see the definition of $\Gamma$-suitable premouse of type $\Delta$ in \cref{suitable_pm}. One can find similar adaptation in \cite{RTr18}.

    \begin{dfn}[{\cite[Definition 2.43]{scales_in_hybrid_mice}}]\label{Lp_stack}\leavevmode
        \begin{enumerate}
            \item Let $\mathcal{M}$ be an $x$-premouse for some transitive $x$. We say that $\mathcal{M}$ is \emph{countably $(n, \omega_1+1)$-iterable (above $\eta$)} if whenever there is an elementary embedding $\pi\colon\overline{\mathcal{M}}\to\mathcal{M}$  with $x\in\mathrm{ran}(\pi)$ (and $\eta\in\mathrm{ran}(\pi)$), then $\overline{\mathcal{M}}$ is $(n, \omega_1+1)$-iterable (above $\pi^{-1}(\eta)$) as a $\pi^{-1}(x)$-premouse.
            \item For a transitive set $x$, $\mathrm{Lp}^\Gamma(x)$ denotes the stack of all countably $(\omega, \omega_1+1)$-iterable $x$-premice such that $\mathcal{M}$ is fully sound and projects to $\omega$.
            \item Let $\mathcal{N}$ be a $x$-premouse for some transitive $x$ and let $\eta = \mathcal{N}\cap\mathrm{Ord}$. Then $\mathrm{Lp}_+^\Gamma(\mathcal{N})$ denotes the stack of all $x$-premice $\mathcal{M}$ extending $\mathcal{N}$ such that either $\mathcal{M} = \mathcal{N}$ or $\mathcal{M}$ is $\eta$-sound, $\rho^{\mathcal{N}}_{n+1}\leq\eta<\rho_n^{\mathcal{N}}$ for some $n<\omega$, and $\mathcal{M}$ is countably $(n, \omega_1+1)$-iterable above $\eta$ as witnessed by iteration strategies whose restrictions to countable iteration trees are in $\Gamma$.
        \end{enumerate}
    \end{dfn}

    \begin{dfn}\label{C_Gamma}
        Let $\Gamma$ be a pointclass. For $x\in\mathbb{R}$, $C_\Gamma(x)$ denotes the set of all $y\in\mathbb{R}$ such that $y$ is $\Delta_{\Gamma}(x)$ in a countable ordinal.\footnote{$\Gamma(x)$ denotes the relativization of $\Gamma$ to $x$ and $\Delta_\Gamma(x)$ denotes the ambiguous part of $\Gamma(x)$, i.e.\ $\Gamma(x)\cap\check{\Gamma}(x)$.} For transitive countable set $x$, $C_\Gamma(x)$ denotes the set of all $y\subset x$ such that for every surjective function $f\colon\omega\to x$, $f^{-1}[y]\in C_{\Gamma}(x_f)$, where $x_f\in\mathbb{R}$ is the code of $x$ via $f$.\footnote{That is, $x_f$ codes $f^{-1}[\mathord{\in}\cap(x\times x)]$ in a recursive way.}
    \end{dfn}
    
    \begin{dfn}[{cf.\ \cite[Definition 5.28]{scales_in_hybrid_mice}}]\label{suitable_pm}
        Let $x\in\mathbb{R}$, $\mathcal{P}$ be a countable $x$-premouse, and $\Gamma$ be a pointclass. We say that $\mathcal{P}$ is \emph{$\Gamma$-suitable} if, letting $\langle\delta_{\alpha}\mid\alpha<\lambda\rangle$ be the increasing enumeration of Woodin cardinals and limits of Woodin cardinals of $\mathcal{P}$, then we have:
        \begin{enumerate}
            \item if $\eta$ is a strong cutpoint of $\mathcal{P}$,\footnote{Recall that $\eta$ is a \emph{strong cutpoint} of a given premouse $\mathcal{P}$ if there is no (partial) extender $E$ on the $\mathcal{P}$-sequence with $\crit(E)\leq\eta<\lh(E)$.} then 
            \[
            \mathcal{P}\vert(\eta^+)^{\mathcal{P}} = \mathrm{Lp}^{\Gamma}_{+}(\mathcal{P}\vert\eta),
            \]
            \item if $\eta\in\mathcal{P}\cap\mathrm{Ord}$ is not Woodin in $\mathcal{P}$, then
            \[
            C_{\Gamma}(\mathcal{P}\vert\eta)\models\text{$\eta$ is not Woodin},
            \]
            \item either
            \begin{enumerate}
                \item $\lambda$ is a successor ordinal and $\mathcal{P}\cap\mathrm{Ord} = \sup_{i<\omega}(\delta_{\lambda - 1}^{+ i})^{\mathcal{P}}$, or
                \item $\lambda$ is a limit ordinal and $\mathcal{P}\cap\mathrm{Ord} = \sup_{\alpha<\lambda}\delta_{\alpha}$.
            \end{enumerate}
            \item (Smallness assumption) for any $\alpha<\lambda$, $\alpha<\delta_{\alpha}$.
        \end{enumerate}
        Now let $\mathcal{P}$ be $\Gamma$-suitable. We write $\langle\delta^{\mathcal{P}}_{\alpha}\mid \alpha<\lambda^{\mathcal{P}}\rangle$ for the increasing enumeration of all Woodin cardinals and limits of Woodin cardinals of $\mathcal{P}$. We also write $\delta_{-1}^{\mathcal{P}}=0$ for convenience. We say that
        \begin{itemize}
            \item $\mathcal{P}$ has \emph{successor type} if $\lambda^{\mathcal{P}}$ is a successor ordinal,
            \item $\mathcal{P}$ has \emph{limit type} if $\lambda^{\mathcal{P}}$ is a limit ordinal, and
            \item $\mathcal{P}$ has \emph{type $\Delta$}\footnote{The letter $\Delta$ indicates the long game we introduced.} if $\lambda^{\mathcal{P}} = \mathcal{P} \cap \mathrm{Ord}$.
        \end{itemize}
        Furthermore, for any $\alpha<\lambda^{\mathcal{P}}$, let $\mathcal{P}(\alpha)$ be the unique $\Gamma$-suitable initial segment $\mathcal{Q}$ of $\mathcal{P}$ such that
        \begin{itemize}
            \item if $\alpha$ is either 0 or a successor ordinal, then $\lambda^{\mathcal{Q}} = \alpha+1$, and
            \item if $\alpha$ is a limit ordinal, then $\lambda^{\mathcal{Q}} = \alpha$.
        \end{itemize}
    \end{dfn}
    
    In the previous literature such as \cite{scales_at_weak_gap,scales_in_hybrid_mice}, the smallness assumption for a $\Gamma$-suitable $\mathcal{P}$ was $\lambda^{\mathcal{P}}\leq\omega$. Even though we relaxed this smallness assumption in \cref{suitable_pm}, basic arguments for $\Gamma$-suitable premice work in our context. Note that in a $\Gamma$-suitable premouse $\mathcal{P}$, every Woodin cardinal is a strong cutpoint even under our smallness assumption.
    
    \begin{dfn}[{\cite[Definition 5.29]{scales_in_hybrid_mice}}]\label{guided_trees}
        Let $\mathcal{P}$ be a countable premouse and $\mathcal{T}$ be a countable normal iteration tree on $\mathcal{P}$. Then $\mathcal{T}$ is \emph{Q-guided} if for each limit $\lambda<\lh(\mathcal{T})$, $\mathcal{Q}:=\mathcal{Q}([0, \lambda]_{\mathcal{T}}, \mathcal{T}\upharpoonright\lambda)$ exists\footnote{See \cite[Definition 2.4]{MSW20} for the definition of $\mathcal{Q}(b, \mathcal{T})$.} and the phalanx $\Phi(\mathcal{T}\upharpoonright\lambda){}^{\frown}\langle\mathcal{Q}, \delta(\mathcal{T})\rangle$ is $(\omega, \omega_1 +1)$-iterable.\footnote{In our context, $\delta(\mathcal{T})$ is always a strong cutpoint of $\mathcal{Q}$, since we only encounter tame premice. In this case, the iterability of $\Phi(\mathcal{T}\upharpoonright\lambda){}^{\frown}\langle\mathcal{Q}, \delta(\mathcal{T})\rangle$ is equivalent to the iterability of $\mathcal{Q}$ above $\delta(\mathcal{T})$.} For a pointclass $\Gamma$, we say that $\mathcal{T}$ is \emph{$\Gamma$-guided} if it is Q-guided, as witnessed by iteration strategies whose codes are in $\Gamma$.
    \end{dfn}
    
    Because $\Gamma$-suitable premice do not have overlapped Woodin cardinals, the standard phalanx comparison argument shows that there is at most one cofinal branch $b$ through $\mathcal{T}$ such that $\mathcal{T}^{\frown}b$ is Q-guided.
    
    \begin{dfn}[{\cite[Definition 5.31]{scales_in_hybrid_mice}}]\label{short_trees}
        Let $\mathcal{P}$ be a $\Gamma$-suitable $x$-premouse for some $x\in\mathbb{R}$ and $\mathcal{T}$ be a $\Gamma$-guided iteration tree on $\mathcal{P}$ of limit length. We say that $\mathcal{T}$ is \emph{$\Gamma$-short} if $\mathcal{Q}(\mathcal{T})$ exists and is initial segment of $\mathrm{Lp}^{\Gamma}_{+}(\mathcal{M}(\mathcal{T}))$.\footnote{See c\cite[Definition 2.4]{MSW20} for the definition of $\mathcal{Q}(\mathcal{T})$. Also, $\mathcal{M}(\mathcal{T})$ denotes the common part model of $\mathcal{T}$ as usual.} Otherwise, we say that $\mathcal{T}$ is \emph{$\Gamma$-maximal}.
    \end{dfn}
    
    \begin{dfn}[{\cite[Definition 5.32]{scales_in_hybrid_mice}}]\label{suitability_strictness}
        Let $\mathcal{P}$ be a $\Gamma$-suitable $x$-premouse for some $x\in\mathbb{R}$.
        Let $\mathcal{T}$ be an iteration tree on $\mathcal{P}$.
        We say that $\mathcal{T}$ is \emph{$\Gamma$-suitability strict} if for every $\alpha<\mathrm{lh}(\mathcal{T})$,
        \begin{enumerate}
            \item If $[0, \alpha]_{\mathcal{T}}$ does not drop then $M^{\mathcal{T}}_{\alpha}$ is $\Gamma$-suitable.
            \item If $[0, \alpha]_{\mathcal{T}}$ drops and there are trees $\mathcal{U}, \mathcal{V}$ such that $\mathcal{T}\upharpoonright\alpha+1 = \mathcal{U}{}^{\frown}\mathcal{V}$, where $\mathcal{U}$ has last model $\mathcal{R}$, $b^{\mathcal{U}}$ does not drop, and $\mathcal{V}$ is based on $[\delta^{\mathcal{R}}_{\beta-1}, \delta^{\mathcal{R}}_{\beta})$ for some $\beta\in\mathcal{R}\cap\mathrm{Ord}$, then there is no $\Gamma$-suitable $\mathcal{Q}\init M^{\mathcal{T}}_{\alpha}$ with $\lambda^{\mathcal{Q}}\geq 1+\beta$.
        \end{enumerate}
        We say that a (partial) iteration strategy $\Sigma$ for $\mathcal{P}$ is \emph{$\Gamma$-suitability strict} if every tree via $\Sigma$ is $\Gamma$-suitability strict.
    \end{dfn}
    
    \begin{dfn}[{\cite[Definition 5.33]{scales_in_hybrid_mice}}]\label{short_tree_iterability}
        Let $\mathcal{P}$ be $\Gamma$-suitable. We say that $\mathcal{P}$ is \emph{$\Gamma$-short tree iterable} if for every normal $\Gamma$-guided tree $\mathcal{T}$ on $\mathcal{P}$,
        \begin{enumerate}
            \item $\mathcal{T}$ is $\Gamma$-suitability strict,
            \item if $\mathcal{T}$ has limit length and is $\Gamma$-short then there is a cofinal wellfounded branch $b$ through $\mathcal{T}$ such that $\mathcal{T}^{\frown}b$ is $\Gamma$-guided,
            \item if $\mathcal{T}$ has successor length, then every one-step putative normal extension of $\mathcal{T}$ is an iteration tree.
        \end{enumerate}
    \end{dfn}
    
    We define (normal, almost, local) $A$-iterability as in \cite{scales_at_weak_gap} with the modification pointed out in \cite[Definition 5.35]{scales_in_hybrid_mice}. For the readers' convenience, we write down the precise definitions below.
    
    \begin{dfn}[{\cite[Definition 1.3]{scales_at_weak_gap}}]\label{local_term_capturing}
        Let $A\subset\mathbb{R}$.
        Let $\mathcal{P}$ be a premouse and let $\nu$ be a cardinal of $\mathcal{P}$.
        We say that \emph{$\mathcal{P}$ captures $A$ at $\nu$} if there is a $\mathrm{Col}(\omega, \nu)$-name $\tau\in\mathcal{P}$ such that $\tau^g = A\cap\mathcal{P}[g]$ for all $\mathcal{P}$-generic $g\subset\mathrm{Col}(\omega, \nu)$. Let $\tau^{\mathcal{P}}_{A, \nu}$ be such a unique standard name. Here, a $\mathrm{Col}(\omega, \nu)$-name $\tau$ is called \emph{standard} if
        \[
        \tau = \{\langle p, \sigma\rangle\mid p\in\mathrm{Col}(\omega, \nu)\land \sigma\subset\mathrm{Col}(\omega, \nu)\times\{\check{n}\mid n\in\omega\land p\Vdash\sigma\in\tau\}\}.
        \]
    \end{dfn}
    
    \begin{dfn}[cf.\ {\cite[Definition 1.6]{scales_at_weak_gap}}]\label{tree_respects_s}
        Let $A\subset\mathbb{R}$ and $\mathcal{P}$ be a $\Gamma$-suitable $y$-premouse for some $y\in\mathbb{R}$, capturing $A$ at all cardinals of $\mathcal{P}$. Let $\mathcal{T}$ be an iteration tree on $\mathcal{P}$. We say that \emph{$\mathcal{T}$ respects $A$} (with respect to $\Gamma$) if $\mathcal{T}$ is $\Gamma$-suitability strict and for any $\alpha<\lh(\mathcal{T})$,
        \begin{enumerate}
            \item if $[0, \alpha]_{\mathcal{T}}$ does not drop, then letting $\mathcal{R} = M^{\mathcal{T}}_{\alpha}$ and $i=i^{\mathcal{T}}_{0, \alpha}$, $\mathcal{R}$ is $\Gamma$-suitable and captures $A$ at all cardinals of $\mathcal{R}$ and
            \[
            i(\tau^{\mathcal{P}}_{A, \nu}) = \tau^{\mathcal{R}}_{A, i(\nu)}
            \]
            holds for all cardinals $\nu$ of $\mathcal{P}$, and
            \item if $\alpha$ is a limit ordinal and $[0, \alpha]_{\mathcal{T}}$ drops, then $\mathcal{T}$ is $\Gamma$-short.
        \end{enumerate}
    \end{dfn}
    
    \begin{dfn}[{\cite[Definition 5.35]{scales_in_hybrid_mice}, \cite[Definition 1.7]{scales_at_weak_gap}}]\label{normally_A-iterability}
        Let $A\subset\mathbb{R}$ and $\mathcal{P}$ be a $\Gamma$-suitable $y$-premouse for some $y\in\mathbb{R}$, capturing $A$ at all cardinals of $\mathcal{P}$. Then $\mathcal{P}$ is \emph{normally $A$-iterable} (with respect to $\Gamma$) if $\mathcal{P}$ is $\Gamma$-short tree iterable and one of the following holds:
        \begin{enumerate}
        \item $\mathcal{P}$ has successor type and whenever $\mathcal{T}$ is a normal $\Gamma$-guided iteration tree on $\mathcal{P}$, then
        \begin{itemize}
            \item $\mathcal{T}$ respects $A$,
            \item if $\mathcal{T}$ has successor length, then every one-step putative normal extension of $\mathcal{T}$ is an iteration tree,
            \item if $\mathcal{T}$ has limit length and is $\Gamma$-short, then there is a cofinal branch $b$ through $\mathcal{T}$ such that $\mathcal{T}^{\frown} b$ is $\Gamma$-guided, and
            \item if $\mathcal{T}$ is $\Gamma$-maximal, then there is a cofinal, wellfounded, non-dropping branch $b$ through $\mathcal{T}$ such that $\mathcal{T}^{\frown} b$ respects $A$.
        \end{itemize}
        \item $\mathcal{P}$ has limit type and for any $\alpha<\lambda^{\mathcal{P}}$, if $\alpha$ is either 0 or a successor, then $\mathcal{P}(\alpha)$ is normally $A$-iterable in the sense of (1).
        \end{enumerate}
    \end{dfn}
    
    \begin{dfn}[{\cite[Definition 1.9]{scales_at_weak_gap}}]\label{almost_A-iterability}
        Let $A\subset\mathbb{R}$ and $\mathcal{P}$ be a $\Gamma$-suitable $y$-premouse for some $y\in\mathbb{R}$, capturing $A$ at all cardinals of $\mathcal{P}$. Then we define the two-player game $G(A, \mathcal{P})$ of length at most $\omega$ as follows:
        \[
        \begin{array}{c|ccccc}
        \mathrm{I}    & \mathcal{T}_0 &        & \mathcal{T}_1 &        & \ldots\\ \hline
        \mathrm{II}   &        & b_0 &        & b_1 &
        \end{array}
        \]
        \begin{itemize}
            \item In the first round, Player I must play a $\Gamma$-maximal $\Gamma$-guided normal iteration tree $\mathcal{T}_0$ on $\mathcal{P}$.
            Then Player II must respond by playing a cofinal, wellfounded, non-dropping branch $b_0$ through $\mathcal{T}_0$ such that $\mathcal{T}_0{}^{\frown} b_0$ respects $A$.
            \item For any $i<\omega$, in the $i+1$-st round, Player I must play a normal, maximal $\Gamma$-guided tree $\mathcal{T}_{i+1}$ on $M^{\mathcal{T}_i}_{b_i}$ and Player II must play a cofinal, wellfounded, non-dropping branch $b_{i+1}$ through $\mathcal{T}_{i+1}$ such that $\mathcal{T}_{i+1}{}^{\frown} b_{i+1}$ respects $A$.
        \end{itemize}
        The first player who violates the rules will lose. Player II wins if and only if the game continues $\omega$ many rounds.
        
        We say that $\mathcal{P}$ is \emph{almost $A$-iterable} (with respect to $\Gamma$) if one of the following holds:
        \begin{enumerate}
            \item $\mathcal{P}$ has successor type and Player II has a winning strategy in $G(A, \mathcal{P})$.
            \item $\mathcal{P}$ has limit type and for any $\alpha<\lambda^{\mathcal{P}}$, if $\alpha$ is either 0 or a successor, then Player II has a winning strategy in $G(A, \mathcal{P}(\alpha))$.
        \end{enumerate}
    \end{dfn}
    
    \begin{dfn}\label{gamma_and_H}
        Let $A\subset\mathbb{R}$ and $\mathcal{P}$ be a $\Gamma$-suitable $y$-premouse for some $y\in\mathbb{R}$, capturing $A$ at all cardinals of $\mathcal{P}$. Let $\alpha<\lambda^{\mathcal{P}}$ be either 0 or a successor ordinal. Then we define
        \begin{align*}
            \gamma_{A, \alpha}^{\mathcal{P}} &= \sup(\mathrm{Hull}^{\mathcal{P}}_1(\{\tau^{\mathcal{P}}_{A, (\delta_{\alpha}^{+i})^{\mathcal{P}}}\mid i<\omega\})\cap\delta_{\alpha}^{\mathcal{P}}),\\
            H_{A, \alpha}^{\mathcal{P}} &= \mathrm{Hull}^{\mathcal{P}}_1(\gamma_{A, \alpha}^{\mathcal{P}}\cup\{\tau^{\mathcal{P}}_{A, (\delta_{\alpha}^{+i})^{\mathcal{P}}}\mid i<\omega\}).
        \end{align*}
        Here, the hulls are uncollapsed $\Sigma_1$ hulls in the language of $y$-premice.
    \end{dfn}
    
    \begin{dfn}[{\cite[Definition 1.11]{scales_at_weak_gap}}]\label{local_A-iterability}
        Let $A\subset\mathbb{R}$ and $\mathcal{P}$ be a $\Gamma$-suitable $y$-premouse for some $y\in\mathbb{R}$, capturing $A$ at all cardinals of $\mathcal{P}$. Let $\alpha<\lambda^{\mathcal{P}}$ be either 0 or a successor ordinal.
        \begin{enumerate}
            \item A finite sequence $s=\langle(\mathcal{P}_i, \mathcal{T}_i, \pi_i)\mid i<n\rangle$ is \emph{$A$-good at $\alpha$} (with respect to $\Gamma$) if $\mathcal{P}_0 = \mathcal{P}$ and the following holds for all $i<n$:
            \begin{enumerate}
                \item $\mathcal{P}_i$ is almost $A$-iterable,
                \item $\mathcal{T}_i$ is a normal iteration tree on $\mathcal{P}_i$ with last model $\mathcal{P}_{i+1}$ that is based on the window $(\delta_{\alpha-1}^{\mathcal{P}_i}, \delta_{\alpha}^{\mathcal{P}_i})$,\footnote{Namely, $\mathcal{T}_i$ is based on $\mathcal{P}_i\vert\delta_{\alpha}^{\mathcal{P}_i}$ and all extenders used in $\mathcal{T}_i$ have critical points above $\delta_{\alpha-1}^{\mathcal{P}_i}$.}
                \item either $\mathcal{T}_i$ is $\Gamma$-guided, or $\mathcal{T}_i$ is of the form $\mathcal{T}_i^{-}{}^{\frown} b_i$ for some $\Gamma$-maximal $\Gamma$-guided iteration tree $\mathcal{T}_i^{-}$ and a cofinal branch $b_i$ through $\mathcal{T}_i^{-}$,
                \item $\mathcal{T}_i$ respects $A$, and
                \item the main branch $b_i$ through $\mathcal{T}_i$ does not drop, and $\pi_i\colon\mathcal{P}_i\to\mathcal{P}_{i+1}$ is the iteration embedding along $b_i$.
            \end{enumerate}
        
            \item We say that an $A$-good sequence $s$ at $\alpha$ gives rise to $\pi\colon\mathcal{P}\to\mathcal{Q}$ if $\mathcal{Q} = \mathcal{P}_n$ and $\pi = \pi_{n-1}\circ\cdots\circ\pi_0$.
        
            \item $\mathcal{P}$ is \emph{locally $A$-iterable at $\alpha$} (with respect to $\Gamma$) if whenever $s$ and $t$ are $A$-good at $\alpha$ and give rise to $\pi\colon\mathcal{P}\to\mathcal{Q}$ and $\sigma\colon\mathcal{P}\to\mathcal{R}$ respectively and $\mathcal{Q}(\alpha) = \mathcal{R}(\alpha)$, then $\pi\upharpoonright H^{\mathcal{P}}_{A, \alpha} = \sigma\upharpoonright H^{\mathcal{P}}_{A, \alpha}$.
        \end{enumerate}
    \end{dfn}

     \begin{dfn}[{\cite[Definition 1.12]{scales_at_weak_gap}}]\label{A-iterability}
        Let $A\subset\mathbb{R}$ and $\mathcal{P}$ be a $\Gamma$-suitable $y$-premouse for some $y\in\mathbb{R}$, capturing $A$ at all cardinals of $\mathcal{P}$. 
            We say that $\mathcal{P}$ is \emph{$A$-iterable} (with respect to $\Gamma$) if one of the following holds:
            \begin{enumerate}
                \item $\mathcal{P}$ has successor type and Player II has a winning strategy in $G(A, \mathcal{P})$ such that whenever $\mathcal{Q}$ is a non-dropping iterate of $\mathcal{P}$ according to the strategy, $\mathcal{Q}$ is locally $A$-iterable at all $\alpha<\lambda^{\mathcal{Q}}$.
                \item $\mathcal{P}$ has limit type and for any $\alpha<\lambda^{\mathcal{P}}$, if $\alpha$ is either 0 or a successor, then $\mathcal{P}(\alpha)$ is $A$-iterable in the sense of (1).
            \end{enumerate}
    \end{dfn}
    
    \begin{rem}
        Almost $A$-iterability is an analogue of $\vec{f}$-iterability in \cite{StW16}, while $A$-iterability defined above is analogous to strong $\vec{f}$-iterability in \cite{StW16}.
    \end{rem}
    
    From now, we write $\Gamma$ for the pointclass $\Sigma^2_1$ in $L(\mathbb{R}, \langle\mu_{\alpha}\mid\alpha<\omega_1\rangle)$. We use the Prikry-type forcing $\mathbb{P}_\Delta$ to generically add a $\Gamma$-suitable premouse of type $\Delta$ over $L(\mathbb{R}, \langle\mu_{\alpha}\mid\alpha<\omega_1\rangle)$. We need to introduce the relevant notation first. Let $a$ be a countable transitive set. Let $x\in\mathbb{R}$ be such that some real recursive in $x$ codes $a$. For any real $z$, we write $\mathcal{P}_z$ for the $a$-premouse coded by $z$. Let $\mathcal{F}^x_a$ be the direct limit system that consists of $\mathcal{P}_z$ for some $z\leq_T x$ that is a $\Gamma$-short tree iterable $\Gamma$-suitable $a$-premouse with $\lambda^{\mathcal{P}_z} = 1$.
    We define $\mathcal{Q}_{a}^{x, -}$ to be the direct limit of the simultaneous comparison and $\{y\in\mathbb{R}\mid y\leq_{T}x\}$-genericity iteration of all $a$-premice in $\mathcal{F}^x_a$. Let $\mathcal{Q}_a^x = \mathrm{Lp}^{\Gamma}_+(\mathcal{Q}_a^{x, -})$ and let $\delta_a^x$ be the largest cardinal of $\mathcal{Q}^x_a$. As $\mathcal{Q}_a^x$ only depends on $a$ and the Turing degree of $x$, we also write $\mathcal{Q}_a^d$ for a Turing degree $d$.
    
    We fix $a$ as above and a sequence $\vec{d}=\langle d_i\mid i<\omega\rangle$ that is $\mathbb{P}_{\Delta}$-generic over $L(\mathbb{R}, \mu_{<\omega_1})$. For each $i<\omega$, let $\alpha_i$ be such that $d_i\in\mathbb{D}_{\alpha_i}$. We define a sequence $\langle\mathcal{Q}^i_{\beta}\mid i<\omega\land\beta\leq\omega^{1+\alpha_i}\rangle$ as follows:
    \begin{align*}
        Q^0_0 &= \mathcal{Q}_a^{d_0(0)},\\
        Q^0_{\beta+1} &= \mathcal{Q}_{Q^0_{\beta}}^{d_0(\beta+1)},\\ 
        Q^0_{\gamma} &= \bigcup_{\beta<\gamma}\mathcal{Q}^0_{\beta}\;\;\;\text{if $\gamma\leq\omega^{1+\alpha_0}$ is limit,}
    \end{align*}
    and for $i<\omega$,
    \begin{align*}
        Q^{i+1}_0 &= \mathcal{Q}_{\mathcal{Q}^1_{\omega^{1+\alpha_i}}}^{d_{i+1}(0)},\\
        Q^{i+1}_{\beta+1} &= \mathcal{Q}_{Q^{i+1}_{\beta}}^{d_{i+1}(\beta+1)},\\
        Q^{i+1}_{\gamma} &= \bigcup_{\beta<\gamma}\mathcal{Q}^{i+1}_{\beta}\;\;\;\text{if $\gamma\leq\omega^{1+\alpha_{i+1}}$ is limit.}
    \end{align*}
    Then we define
    \[
    M^{\vec{d}}(a) = \bigcup\{\mathcal{Q}^i_{\beta}\mid i<\omega \land \beta\leq\omega^{1+\alpha_i}\}.
    \]
    Also, for any $i<\omega$ and $\beta\leq\omega^{1+\alpha_i}$, we define
    \[
    \delta^i_{\beta} = \begin{cases}
        \text{the largest cardinal of }\mathcal{Q}^i_{\beta} & \text{if $\beta$ is either 0 or successor},\\
        \ord\cap\mathcal{Q}^i_{\beta} & \text{if $\beta$ is limit}.
    \end{cases}
    \]
    The following proposition can be shown in an analogous way as the proofs of \cite[Subsection 6.6]{StW16} and \cite[Lemma 2.6]{Tr15}. Here, one needs \cref{mouse_capturing}, in particular, $\mathsf{MC}$ in $L(\mathbb{R}, \mu_{<\omega_1})$.
    
    \begin{prop}
        Suppose that
        \[
        L(\mathbb{R}, \mu_{<\omega_1})\models\mathsf{AD} + \forall\alpha<\omega_1\,(\mu_{\alpha}\text{ is an ultrafilter on }[\power_{\omega_1}(\mathbb{R})]^{\omega^\alpha})
        \]
        and let $\Gamma$ be the pointclass $\Sigma^2_1$ in $L(\mathbb{R}, \mu_{<\omega_1})$. Let $a$ be countable transitive. Then for any sequence $\vec{d}$ that is $\mathbb{P}_{\Delta}$-generic over $L(\mathbb{R}, \mu_{<\omega_1})$, $M^{\vec{d}}(a)$ is a $\Gamma$-suitable premouse of type $\Delta$.
    \end{prop}
    
    \subsection{Main arguments}

    Since variants of the following lemma were proved in many papers (e.g.\ \cite[Section 3]{Tr15}, \cite[Section 7]{StW16}, \cite[Lemma 1.12.1]{scales_at_weak_gap} and \cite[Lemma 5.38]{scales_in_hybrid_mice}), we only give a proof outline here.

    \begin{lem}\label{capturing_OD_sets}
        Assume that
        \[
        L(\mathbb{R}, \mu_{<\omega_1})\models\mathsf{AD} + \forall\alpha<\omega_1\,(\mu_{\alpha}\text{ is an ultrafilter on }[\power_{\omega_1}(\mathbb{R})]^{\omega^\alpha}).
        \]
        Let $\Gamma$ be the pointclass $\Sigma^2_1$ in $L(\mathbb{R}, \mu_{<\omega_1})$. Let $A\subset\mathbb{R}$ be ordinal definable from some $x\in\mathbb{R}$ in $L(\mathbb{R}, \mu_{<\omega_1})$. Then for a cone of $y\geq_T x$, there is an $A$-iterable $\Gamma$-suitable $y$-premouse of type $\Delta$.
    \end{lem}
    \begin{proof}
        We denote the structure $\langle L(\mathbb{R}, \mu_{<\omega_1}), \in, \mu_{<\omega_1}\rangle$ by $M$ and for any ordinal $\alpha$, $M\vert\alpha$ denotes $\langle L_\alpha(\mathbb{R}, \mu_{<\omega_1}), \in, \mu_{<\omega_1}\cap L_\alpha(\mathbb{R}, \mu_{<\omega_1})\rangle$. Throughout the proof, we work in $M$. Note that $M\models\mathsf{AD}^+ + \mathsf{DC} + \mathsf{MC}$.
        
        We say that $A\subset\mathbb{R}$ is \emph{$\Gamma$-bad} if it is a counterexample to the statement of \cref{capturing_OD_sets}, where $\Gamma = (\Sigma^2_1)^M$. Suppose toward a contradiction that there is a $\Gamma$-bad set of reals that is ordinal definable from a real $x$ (in $M$).
        Then there is $\xi>\Theta$ such that
        \begin{align*}
        M\vert\xi\models~ & \mathsf{ZF}^{-}+\mathsf{AD}^+ +\mathsf{DC} + \mathsf{MC} \\
        & + \text{``there is a $\Gamma$-bad sets of reals that is ordinal definable from $x$.''} 
        \end{align*}
        By $\Sigma_1$ reflection in $L(\mathbb{R}, \mu_{<\omega_1})$ (\cref{Sigma_1_reflection}), there is a $\xi<\delta^2_1$ and a pointclass $\overline{\Gamma}$ such that letting $[\alpha, \beta]$ be the $\Sigma_1$-gap of $M$ including $\xi$ and $\overline{\Gamma}=(\Sigma^2_1)^{M\vert\alpha}$, then
        \begin{align*}
        M\vert\xi\models~ & \mathsf{ZF}^{-}+\mathsf{AD}^+ + \mathsf{DC} + \mathsf{MC}\\
         & + \text{``there is a $\overline{\Gamma}$-bad sets of reals that is ordinal definable from $x$,''}
        \end{align*}
        where we define $\overline{\Gamma}$-badness in an analogous way. We may assume that $\xi$ is least such. Then $\beta = \xi+1$ and $[\alpha, \xi+1]$ is a (weak) $\Sigma_1$-gap of $M$. Fix $A\in\power(\mathbb{R})\cap M\vert\xi$ such that
        \[
        M\vert\xi\models\text{``$A$ is $\overline{\Gamma}$-bad and ordinal definable from $x$.''}
        \]
        Note that $A$ is truely $\overline{\Gamma}$-bad (in $M$) by the absoluteness of $A$-iterability. We will reach a contradiction by arguing that $A$ is not $\overline{\Gamma}$-bad.
    
        By the result in \cite{scales_in_K(R)}, there is a self-justifying system $\mathcal{A}=\langle A_i\mid i<\omega\rangle$ sealing the envelope of $\overline{\Gamma}$ such that $A_0 = A$. Note that each $A_i$ is in $M\vert\beta$. Let $z_0 \leq_{T} z_1 \leq_{T} z_2$ be reals satisfying the following:
        \begin{itemize}
            \item $x\leq_T z_0$ and in $M\vert\xi$, for any $z\geq_{T}z_0$, there is no $A$-iterable $\overline{\Gamma}$-suitable $z$-premouse of type $\Delta$. (We use $\mathsf{AD}$ in $M\vert\xi$ to find such a $z_0$.)
            \item For any $i<\omega$, there is $\zeta<\Theta^{M\vert\beta}$ such that $A_i$ is definable from $z_1$ over $M\vert\zeta$.
            \item Since $C_{\overline{\Gamma}}(z_1)\subsetneq C_{\Gamma}(z_1)$, $\mathsf{MC}$ (in $M$) implies that there is a $z_1$-premouse $\mathcal{M}\triangleleft\mathrm{Lp}^{\Gamma}(z_1)$ such that $\rho_{\omega}^{\mathcal{M}} = \omega$ and $\mathcal{M}\not\mathrel{\triangleleft}\mathrm{Lp}^{\overline{\Gamma}}(z_1)$.
            Let $\mathcal{M}$ be the least such initial segment and let $z_2\geq_T z_1$ code $\mathcal{M}$.
        \end{itemize}
        Furthermore, let $\Sigma_{\mathcal{M}}$ be an $(\omega, \omega_1)$-iteration strategy for $\mathcal{M}$ in $\Gamma$, which can be extended to $(\omega, \omega_1+1)$-iteration strategy since $\mathsf{AD}$ holds in $M$.
        Let $\Gamma'\subsetneq\Delta_{\Gamma}$ be a good pointclass such that $\Delta_{\overline{\Gamma}}\subset\Gamma'$ and $\mathcal{A}, \Sigma_{\mathcal{M}}\in \Delta_{\Gamma'}$.
        By \cite[Theorem 10.3]{St08}, there is a real $z \geq_{T} z_2$ such that $(\mathcal{N}^*_z, \delta_z, \Sigma_z)$ Suslin-co-Suslin captures $\mathcal{A}$ and the code of $\Sigma_{\mathcal{M}}$.
        Here, $\mathcal{N}^*_z$ is an iterable coarse structure with $z\in\mathcal{N}^*_z$, $\delta_z$ is the unique Woodin cardinal of $\mathcal{N}^*_z$, and $\Sigma_z \in \Delta_{\Gamma'}$ is an iteration strategy for $\mathcal{N}^*_z$.
        
        Let $\kappa_z$ be the least ${<}\delta_z$-strong cardinal of $\mathcal{N}^*_z$.
        In the proof of \cite[Theorem 3.13]{Tr15} and \cite[Lemma 2.4]{SarSt15}, it is shown that in $\mathcal{N}^*_z$, $\kappa_z$ is a limit of ordinals $\eta$ such that
        \[
        \mathrm{Lp}^{\overline{\Gamma}}(V_{\eta}^{\mathcal{N}^*_z})\models\text{``$\eta$ is a Woodin cardinal.''}
        \]
        Let $\langle\eta_i\mid i<\kappa_z\rangle$ be the increasing enumeration of all such ordinals $\eta$ and limits of these. Let $\lambda<\kappa_z$ be the least ordinal such that $\eta_{\xi} = \sup_{i<\xi}\eta_i$. Now we define a sequence $\langle\mathcal{N}_i\mid i<\lambda\rangle$ of premice as follows:
        \begin{itemize}
            \item $\mathcal{N}_0$ is the output of the maximal fully backgrounded construction in $V_{\eta_0}^{\mathcal{N}^*_z}$.
            \item for any $i<\lambda$, $\mathcal{N}_{i+1}$ is the output of the maximal fully backgrounded construction over $\mathcal{N}_i$ in $V_{\eta_{i+1}}^{\mathcal{N}^*_z}$ (using extenders with critical points above $\eta_i$).
            \item for any limit $i<\lambda$, $\mathcal{N}_i = \bigcup_{j<i}\mathcal{N}_j$.
        \end{itemize}
        Using the Prikry-type forcing $\mathbb{P}_{\Delta}$ introduced in the previous subsection, one can show analogously to the proof of \cite[Theorem 3.13]{Tr15} that $\mathcal{N}_{\lambda}$ is a $\overline{\Gamma}$-suitable premouse of type $\Delta$ as witnessed by $\langle\eta_i\mid i<\lambda\rangle$. Moreover, there is some $\mathcal{W}^*\triangleleft\mathrm{Lp}^{\Gamma}(\mathcal{N}_{\lambda})$ such that $\rho_{\omega}^{\mathcal{W}^*}<\mathcal{N}_{\lambda}\cap\mathrm{Ord}$. Let $\mathcal{W}^*$ be the least such initial segment. Then the proof of \cite[Claim 7.5]{StW16} gives the following.
    
        \begin{claim}\label{claim_on_derived_models}
            $\mathcal{W}^* = \mathcal{J}_1(\mathcal{W})$\footnote{This means that $\mathcal{W}^*$ is the rudimentary closure of $\mathcal{W}\cup\{\mathcal{W}\}$.} for some $\mathcal{W}$ such that the derived $\Delta$-Solovay model of $\mathcal{W}$ at some $\lambda$ satisfies ``$\mathsf{ZF}^{-}+$there is a $\overline{\Gamma}$-bad sets of reals'' and no proper initial segment of $\mathcal{W}$ has this property.
        \end{claim}
    
        Let $\mathcal{N}$ be the pointwise definable hull of $\mathcal{W}$ and let $\Lambda$ be its iteration strategy indcued by pulling back an iteration strategy for $\mathcal{W}$ in $\Gamma$. For a given premouse, its $\Delta$-suitable initial segment means the unique initial segment that is a $\Gamma$-suitable premouse of type $\Delta$.
    
        \begin{claim}\label{nonbadness_of_A}
            The $\Delta$-suitable initial segment $\mathcal{P}$ of $\mathcal{N}$ is $A$-iterable as witnessed by $\Lambda$.
        \end{claim}
        \begin{proof}
        We only show that any normal iteration tree based on $\mathcal{P}$ according to $\Lambda$ respects $A$ as the rest of the proof do not need any new idea.
        Suppose that $\mathcal{T}$ is a normal tree on $\mathcal{N}$ with last model $\mathcal{N}'$ and letting $i\colon\mathcal{N}\to\mathcal{N}'$ be the iteration map and $\mathcal{Q}$ be the $\Delta$-suitable initial segment of $\mathcal{N}'$,
        \[
        i(\tau^{\mathcal{P}}_{A, \nu})\neq \tau^{\mathcal{Q}}_{A, i(\nu)}
        \]
        for some cardinal $\nu$ of $\mathcal{P}$.
        Let $\pi\colon\overline{M}\to M\vert\xi$ be an elementary map such that $\overline{M}$ is countable transitive and $\mathcal{N}, \mathcal{T}, S, T\in\mathrm{ran}(\pi)$, where $S$ and $T$ are trees projecting to a universal $\check{\Gamma}$ set and its complement respectively.
        We may also assume that $A$ is ordinal definable from a real in $\overline{M}$.
        Now take $\mathbb{R}^{\overline{M}}$-genericity iterates of $\mathcal{N}$ and $\mathcal{N}'$ using Woodin cardinals above $\nu$ and $i(\nu)$ respectively. We can arrange that there is an elementary embedding $j\colon\mathcal{W}\to\mathcal{W}'$ between the last models $\mathcal{W}, \mathcal{W}'$ of the $\mathbb{R}^{\overline{M}}$-genericity iterations of $\mathcal{N}$ and $\mathcal{N}'$ (see the proof of \cite[Lemma 7.7]{StW16}).
        Then the derived $\Delta$-Solovay models of $\mathcal{W}$ and $\mathcal{W}'$ are $\overline{M}$ by \cref{claim_on_derived_models}, so the elementarity of $j$ implies that $j(\tau^{\mathcal{W}}_{A, \nu})=\tau^{\mathcal{W}'}_{A, j(\nu)}$.
        Together with the agreement between $i$ and $j$, it follows that
        \[
        i(\tau^{\mathcal{P}}_{A, \nu}) = j(\tau^{\mathcal{W}}_{A, \nu}) = \tau^{\mathcal{W}'}_{A, j(\nu)} = \tau^{\mathcal{Q}}_{A, i(\nu)},
        \]
        which contradicts the choice of $\mathcal{Q}$ and $\nu$.
        \end{proof}

        \cref{nonbadness_of_A} implies that $A$ is not $\overline{\Gamma}$-bad, which contradicts the choice of $A$. This completes the proof of \cref{capturing_OD_sets}.
    \end{proof}

    \begin{dfn}\label{vecA-iterability}
        Let $\mathcal{P}$ be a $\Gamma$-suitable $y$-premouse of type $\Delta$ for some $y\in\mathbb{R}$. We inductively define $\eta_i^{\mathcal{P}}$ for $i<\omega$ by $\eta_0^{\mathcal{P}} = \delta_0^{\mathcal{P}}$ and $\eta_{i+1}^{\mathcal{P}} = \delta_{\eta_i^{\mathcal{P}} + 1}^{\mathcal{P}}$. (Note that $\mathcal{P}\cap\ord = \sup_{i<\omega}\eta_i^{\mathcal{P}}$.) Also, let $\vec{A}=\langle A_i\mid i<\omega\rangle$ be a sequence of sets of reals. We say that $\mathcal{P}$ is \emph{$\vec{A}$-iterable} if for all $k<\omega$, $\mathcal{P}\vert\eta_k^{\mathcal{P}}$ is $\left(\bigoplus_{i<k}A_i\right)$-iterable.
    \end{dfn}

    The following is an easy consequence of \cref{capturing_OD_sets}:

    \begin{cor}\label{simultaneous_comparison}
        Assume that
        \[
        L(\mathbb{R}, \mu_{<\omega_1})\models\mathsf{AD} + \forall\alpha<\omega_1\,(\mu_{\alpha}\text{ is an ultrafilter on }[\power_{\omega_1}(\mathbb{R})]^{\omega^\alpha}).
        \]
        Let $\Gamma$ be the pointclass $\Sigma^2_1$ in $L(\mathbb{R}, \mu_{<\omega_1})$. Let $\vec{A}=\langle A_i\mid i<\omega\rangle$ be an $\omega$-sequence of sets of reals that are ordinal definable from $x\in\mathbb{R}$ in $M$. Then for a cone of $y\geq_T x$, there is an $\vec{A}$-iterable $\Gamma$-suitable $y$-premouse of type $\Delta$.
    \end{cor}

    The next lemma follows from the argument in \cite{scales_in_K(R)}.

    \begin{lem}\label{existence_of_sjs}
        Assume that
        \[
        L(\mathbb{R}, \mu_{<\omega_1})\models\mathsf{AD} + \forall\alpha<\omega_1\,(\mu_{\alpha}\text{ is an ultrafilter on }[\power_{\omega_1}(\mathbb{R})]^{\omega^\alpha})
        \]
        and the sharp for $L(\mathbb{R}, \mu_{<\omega_1})$ exists. Then $\power(\mathbb{R})\cap L(\mathbb{R}, \mu_{<\omega_1})$ is weakly scaled.
    \end{lem}

    \begin{rem}
        The recent work of Aguilera and the first author in \cite{AG26} can be naturally adopted to show that under the assumption of \cref{existence_of_sjs},
        \[
        \power(\mathbb{R})\cap L(\mathbb{R}, \mu_{<\omega_1}) = \Game_{\Delta}({<}\omega^2\mathchar`-\mathbf{\Pi}^1_1),
        \]
        where $\Game_{\Delta}$ is the game quantifier corresponding to $G_\Delta$. Then the proof of Martin's scale propagation theorem \cite{Ma08} implies that $\power(\mathbb{R})\cap L(\mathbb{R}, \mu_{<\omega_1})$ is weakly scaled, so this gives another proof of \cref{existence_of_sjs}. We leave details to the reader.
    \end{rem}

    \begin{lem}\label{limit_branch_construction}
        Assume that
        \[
        L(\mathbb{R}, \mu_{<\omega_1})\models\mathsf{AD} + \forall\alpha<\omega_1\,(\mu_{\alpha}\text{ is an ultrafilter on }[\power_{\omega_1}(\mathbb{R})]^{\omega^\alpha})
        \]
        and the sharp for $L(\mathbb{R}, \mu_{<\omega_1})$ exists. Let $\Gamma$ be the pointclass $\Sigma^2_1$ in $L(\mathbb{R}, \mu_{<\omega_1})$. Then for a Turing cone of $y$, there is a $\Gamma$-suitable $y$-premouse of type $\Delta$ that is normally $\omega_1$-iterable.
    \end{lem}
    \begin{proof}
    By \cref{existence_of_sjs}, there is a self-justifying system $\vec{A} = \langle A_i\mid i<\omega\rangle$ sealing the envelope of $\Gamma$.
    Let $x\in\mathbb{R}$ be such that each $A_k$ is ordinal definable from $x$ in $M$.
    Write $A'_j = \bigoplus_{i<j}A_i$ for any $j<\omega$.
    By \cref{simultaneous_comparison}, there is a $\vec{A}$-iterable $\Gamma$-suitable $y$-premouse $\mathcal{P}$ of type $\Delta$ for some $y\geq_T x$.
    Let
    \[
    \mathcal{Q}=\mathrm{cHull}^{\mathcal{P}}_1(\eta_0^{\mathcal{P}}\cup\{\tau^{\mathcal{P}}_{A'_j, (\eta_j^{\mathcal{P}+i})^{\mathcal{P}}}\mid i, j<\omega\})
    \]
    and let $\pi\colon\mathcal{Q}\to\mathcal{P}$ be the uncollapse map.
    By the condensation property of a self-justifying system (cf.\ \cite[Theorem 5.4.3]{SchSt}), $\mathcal{Q}$ is a $\Gamma$-suitable $y$-premouse of type $\Delta$ and
    \[
    \pi(\tau^{\mathcal{Q}}_{A'_j, (\eta_j^{\mathcal{Q}+i})^{\mathcal{Q}}}) = \tau^{\mathcal{P}}_{A'_j, (\eta_j^{\mathcal{P}+i})^{\mathcal{P}}}
    \]
    for all $i, j<\omega$.
    Moreover, $\mathcal{Q}$ is $\vec{A}$-iterable witnessed by the pullback strategies of the $\mathcal{P}\vert\eta_k^{\mathcal{P}}$'s and has the following key property: for all $k<\omega$,
    \begin{enumerate}
        \item $\sup_{j<\omega}\gamma_{A'_j, \eta_{k-1}^{\mathcal{Q}}+1}^{\mathcal{Q}} = \eta_k^{\mathcal{Q}}$,\footnote{Recall that $\eta_k^{\mathcal{Q}} = \delta^{\mathcal{Q}}_{\eta_{k-1}^{\mathcal{Q}}+1}$.} and
        \item $\mathcal{Q} = \mathrm{Hull}^{\mathcal{Q}}_1(\eta_k^{\mathcal{Q}}\cup\{\tau^{\mathcal{Q}}_{A'_j, (\eta_j^{\mathcal{Q} +i})^{\mathcal{Q}}}\mid i, j<\omega\})$.
    \end{enumerate}
    Following \cite{scales_at_weak_gap}, we call the first property \emph{$k$-stability} and the second property \emph{$k$-soundness}.
    
    For each $k<\omega$, let $\Sigma_k$ be an $A'_k$-iteration strategy for $\mathcal{Q}\vert\eta_k^{\mathcal{Q}}$.
    We define the desired iteration strategy $\Lambda$ for $\mathcal{Q}$ as follows:
    Let $\mathcal{T}$ be a putative normal iteration tree of countable length on $\mathcal{Q}$ that is based on $(\eta_{j-1}^{\mathcal{Q}}, \eta_j^{\mathcal{Q}})$.

    First, assume that either $\mathcal{T}$ has successor length or $\mathcal{T}$ is $\Gamma$-short.
    Then $\mathcal{T}$ is according to $\Sigma_k$ for all $k\geq j$ (when $\mathcal{T}$ is regarded as an iteration tree on $\mathcal{Q}\vert\eta_k^{\mathcal{Q}}$).
    It follows that if $\mathcal{T}$ has successor length, its last model is wellfounded and that if $\mathcal{T}$ has limit length, then for all $k\geq j$, $\Sigma_k$ chooses the same cofinal branch $b$ of $\mathcal{T}$ and $M_b^{\mathcal{T}}$ (which is an iterate of $\mathcal{Q}$) is wellfounded.\footnote{Here, we use that $\mathcal{Q}\cap\ord = \sup_{k<\omega}\eta_k^{\mathcal{Q}}$.}
    So we define $\Lambda(\mathcal{T})=b$.
    
    Now assume instead that $\mathcal{T}$ is $\Gamma$-maximal.
    Let $b_k = \Sigma_k(\mathcal{T})$ and let $\mathcal{Q}_k = M_{b_k}^{\mathcal{T}}$.
    Then we define $b=\Lambda(\mathcal{T})$ as the ``limit'' of the branches $b_k$ by
    \[
    \eta\in b \iff \exists k\forall l\geq k (\eta\in b_l).
    \]
    The proof of \cite[Theorem 5.4.14]{SchSt} implies that $b$ is a cofinal well-founded branch of $\mathcal{T}$ (that respects all $A_i$'s).
    The $k$-stability is used to show the cofinality of $b$ and the $k$-soundness is used to show the wellfoundedness of $b$.
    Moreover, one can show that $b$ respects all $A_i$'s, $M_b^{\mathcal{T}}$ is still $\vec{A}$-iterable, $k$-stable, and $k$-sound.
    
    One can easily extend the above argument to define $\Lambda(\mathcal{T})$ for any normal iteration tree $\mathcal{T}$ of countable length on $\mathcal{Q}$.\footnote{The proof does not help us to define an iteration strategy for $\mathcal{Q}$ acting on non-normal iteration trees; let $\mathcal{T}$ be a maximal tree based on $(\eta_{k-1}^{\mathcal{Q}}, \eta_k^{\mathcal{Q}})$, let $b=\Lambda(\mathcal{T})$, and let $\mathcal{R}=M_b^{\mathcal{T}}$.
    Let $\mathcal{U}$ be a maximal tree based on $(\eta_{k'-1}^{\mathcal{R}}, \eta_{k'}^{\mathcal{R}})$ for some $k'<k$.
    Since $\mathcal{R}$ is not $k'$-sound, we cannot use the above proof to define $\Lambda(\mathcal{T}{}^{\frown}b{}^{\frown}\mathcal{U})$.}
    \end{proof}

    As in \cite[Section 2]{FNS10}, \cref{limit_branch_construction} implies \cref{mouse_from_determinacy}. Finally, \cref{mainthm} follows from \cref{LGD_from_LC}, \cref{target_DetSol}, and \cref{mouse_from_determinacy}.

    \newcommand{\etalchar}[1]{$^{#1}$}


\begin{thebibliography}{MSW20}
    
    \bibitem[AG]{AG26}
    Juan~P. Aguilera and Takehiko Gappo.
    \newblock Games of decomposable length and inner models.
    \newblock In preparation.
    
    \bibitem[AM20]{AM20}
    Juan~P. Aguilera and Sandra M\"uller.
    \newblock The consistency strength of long projective determinacy.
    \newblock {\em J. Symb. Log.}, 85(1):338--366, 2020.
    
    \bibitem[Cha17]{ChanNotes}
    William Chan.
    \newblock Jensen's model existence theorem.
    \newblock Available at \url{https://williamchan-math.github.io/}, 2017.
    
    \bibitem[FNS10]{FNS10}
    Gunter Fuchs, Itay Neeman, and Ralf Schindler.
    \newblock A criterion for coarse iterability.
    \newblock {\em Arch. Math. Logic}, (49):447--467, 2010.
    
    \bibitem[Fri73]{Fri73}
    Harvey Friedman.
    \newblock Countable models of set theories.
    \newblock In {\em Cambridge {S}ummer {S}chool in {M}athematical {L}ogic ({C}ambridge, 1971)}, volume Vol. 337 of {\em Lecture Notes in Math.}, pages 539--573. Springer, Berlin-New York, 1973.
    
    \bibitem[Fri71]{Fri71}
    Harvey~M. Friedman.
    \newblock Higher set theory and mathematical practice.
    \newblock {\em Ann. Math. Logic}, 2(3):325--357, 1970/71.
    
    \bibitem[Har78]{Har78}
    Leo Harrington.
    \newblock Analytic determinacy and {$0\sp{\sharp }$}.
    \newblock {\em J. Symbolic Logic}, 43(4):685--693, 1978.
    
    \bibitem[Jen10]{JensenAdmissible}
    Ronald Jensen.
    \newblock Admissible sets.
    \newblock Available at \url{https://www.math.uni-bonn.de/~raesch/jensen/}, 2010.
    
    \bibitem[Jen14]{JensenSubcomplete}
    Ronald Jensen.
    \newblock Subcomplete forcing and {$\mathscr{L}$}-forcing.
    \newblock In {\em {$E$}-recursion, forcing and {$C^*$}-algebras}, volume~27 of {\em Lect. Notes Ser. Inst. Math. Sci. Natl. Univ. Singap.}, pages 83--182. World Sci. Publ., Hackensack, NJ, 2014.
    
    \bibitem[Ket11]{Ket11}
    Richard Ketchersid.
    \newblock More structural consequences of {${\rm AD}$}.
    \newblock In {\em Set theory and its applications}, volume 533 of {\em Contemp. Math.}, pages 71--105. Amer. Math. Soc., Providence, RI, 2011.
    
    \bibitem[KW10]{KW10}
    Peter Koellner and W.~Hugh Woodin.
    \newblock Large cardinals from determinacy.
    \newblock In {\em Handbook of set theory. {V}ols. 1, 2, 3}, pages 1951--2119. Springer, Dordrecht, 2010.
    
    \bibitem[Mar08]{Ma08}
    Donald~A. Martin.
    \newblock {The Real Game Quantifier Propagates Scales}.
    \newblock In A.~S. Kechris, B.~L{\"owe}, and J.~R. Steel, editors, {\em Games, Scales and Suslin Cardinals, The Cabal Seminar, Volume I}. Cambridge University Press, 2008.
    
    \bibitem[Mar70]{meas_and_analytic_games}
    Donald~A. Martin.
    \newblock Measurable cardinals and analytic games.
    \newblock {\em Fund. Math.}, 66:287--291, 1969/70.
    
    \bibitem[MR73]{SummerSchoolBook73}
    A.~R.~D. Mathias and H.~Rogers, editors.
    \newblock {\em Cambridge {S}ummer {S}chool in {M}athematical {L}ogic (held in {C}ambridge, {E}ngland, {A}ugust 1--21, 1971)}, volume Vol. 337 of {\em Lecture Notes in Mathematics}.
    \newblock Springer-Verlag, Berlin-New York, 1973.
    
    \bibitem[MS89]{MS89}
    Donald~A. Martin and John~R. Steel.
    \newblock A proof of projective determinacy.
    \newblock {\em J. Amer. Math. Soc.}, 2(1):71--125, 1989.
    
    \bibitem[MS08a]{MS08}
    D.~A. Martin and J.~R. Steel.
    \newblock {The Extent of Scales in $L(\mathbb{R})$}.
    \newblock In A.~S. Kechris, B.~L{\"owe}, and J.~R. Steel., editors, {\em Games, Scales, and Suslin Cardinals. The Cabal Seminar, Volume I}, pages 110--120. The Association of Symbolic Logic, 2008.
    
    \bibitem[MS08b]{MaSt08}
    Donald~A. Martin and John~R. Steel.
    \newblock {The Extent of Scales in $L(\mathbb{R})$}.
    \newblock In Alexander~S. Kechris, Benedikt L\"{o}we, and John~R. Steel, editors, {\em Games, Scales, and Suslin Cardinals. The Cabal Seminar, Volume I}, pages 110--120. The Association of Symbolic Logic, 2008.
    
    \bibitem[MSW20]{MSW20}
    Sandra M\"uller, Ralf Schindler, and W.~Hugh Woodin.
    \newblock Mice with finitely many {W}oodin cardinals from optimal determinacy hypotheses.
    \newblock {\em J. Math. Log.}, 20:1950013, 118, 2020.
    
    \bibitem[Nee95]{OptDet1}
    Itay Neeman.
    \newblock Optimal proofs of determinacy.
    \newblock {\em Bull. Symbolic Logic}, 1(3):327--339, 1995.
    
    \bibitem[Nee04]{Neemanbook}
    Itay Neeman.
    \newblock {\em The determinacy of long games}, volume~7 of {\em De Gruyter Series in Logic and its Applications}.
    \newblock Walter de Gruyter GmbH \& Co. KG, Berlin, 2004.
    
    \bibitem[Nee06]{adm}
    Itay Neeman.
    \newblock Determinacy for games ending at the first admissible relative to the play.
    \newblock {\em J. Symbolic Logic}, 71(2):425--459, 2006.
    
    \bibitem[Nee07]{Games_of_length_omega_1}
    Itay Neeman.
    \newblock Games of length {$\omega_1$}.
    \newblock {\em J. Math. Log.}, 7(1):83--124, 2007.
    
    \bibitem[RT18]{RTr18}
    Daniel Rodr\'iguez and Nam Trang.
    \newblock {$L({\Bbb R},\mu)$} is unique.
    \newblock {\em Adv. Math.}, 324:355--393, 2018.
    
    \bibitem[Sol21]{Independence_DC}
    Robert~M. Solovay.
    \newblock The independence of {$\mathsf{DC}$} from {$\mathsf{AD}$}.
    \newblock In {\em Large cardinals, determinacy and other topics. {T}he {C}abal {S}eminar. {V}ol. {IV}}, volume~49 of {\em Lect. Notes Log.}, pages 66--95. Cambridge Univ. Press, Cambridge, 2021.
    
    \bibitem[SS14]{SchSt}
    R.~Schindler and J.~R. Steel.
    \newblock The core model induction.
    \newblock Available at \url{https://ivv5hpp.uni-muenster.de/u/rds/}, 2014.
    
    \bibitem[SS15]{SarSt15}
    Grigor Sargsyan and John Steel.
    \newblock The mouse set conjecture for sets of reals.
    \newblock {\em J. Symb. Log.}, 80(2):671--683, 2015.
    
    \bibitem[ST16]{scales_in_hybrid_mice}
    Farmer Schlutzenberg and Nam Trang.
    \newblock Scales in hybrid mice over $\mathbb{R}$.
    \newblock Available at \url{https://arxiv.org/abs/1210.7258}, 2016.
    
    \bibitem[Ste07]{St07}
    J.~R. Steel.
    \newblock A stationary-tower-free proof of the derived model theorem.
    \newblock In {\em Advances in logic}, volume 425 of {\em Contemp. Math.}, pages 1--8. Amer. Math. Soc., Providence, RI, 2007.
    
    \bibitem[Ste08a]{scales_in_K(R)}
    J.~R. Steel.
    \newblock {Scales in $K(\mathbb{R})$}.
    \newblock In A.~S. Kechris, B.~L{\"owe}, and J.~R. Steel, editors, {\em Games, Scales and Suslin Cardinals, The Cabal Seminar, Volume I}. Cambridge University Press, 2008.
    
    \bibitem[Ste08b]{scales_at_weak_gap}
    J.~R. Steel.
    \newblock Scales in {$K(\mathbb{R})$} at the end of a weak gap.
    \newblock {\em J. Symbolic Logic}, 73(2):369--390, 2008.
    
    \bibitem[Ste08c]{St08}
    John~R. Steel.
    \newblock Derived models associated to mice.
    \newblock In {\em Computational prospects of infinity. {P}art {I}. {T}utorials}, volume~14 of {\em Lect. Notes Ser. Inst. Math. Sci. Natl. Univ. Singap.}, pages 105--193. World Sci. Publ., Hackensack, NJ, 2008.
    
    \bibitem[Ste09]{St09}
    J.~R. Steel.
    \newblock The derived model theorem.
    \newblock In {\em Logic {C}olloquium 2006}, volume~32 of {\em Lect. Notes Log.}, pages 280--327. Assoc. Symbol. Logic, Chicago, IL, 2009.
    
    \bibitem[Ste16]{Woodin_mouse_sets}
    John~R. Steel.
    \newblock A theorem of {W}oodin on mouse sets.
    \newblock In {\em Ordinal definability and recursion theory: {T}he {C}abal {S}eminar. {V}ol. {III}}, volume~43 of {\em Lect. Notes Log.}, pages 243--256. Assoc. Symbol. Logic, Ithaca, NY, 2016.
    
    \bibitem[SW16]{StW16}
    J.~R. Steel and W.~H. Woodin.
    \newblock {HOD as a core model}.
    \newblock In A.~S. Kechris, B.~Löwe, and J.~R. Steel, editors, {\em {Ordinal definability and recursion theory}}, pages 257--346. 2016.
    
    \bibitem[Tra13]{TrangThesis}
    Nam Trang.
    \newblock {\em Generalized Solovay measures, the HOD analysis, and the core model induction}.
    \newblock PhD thesis, UC Berkeley, 2013.
    
    \bibitem[Tra14]{DetLRmu}
    Nam Trang.
    \newblock Determinacy in {$L(\mathbb{R}, \mu)$}.
    \newblock {\em Journal of Mathematical Logic}, 14(01):1450006, 2014.
    
    \bibitem[Tra15]{Tr15}
    Nam Trang.
    \newblock Structure theory of {$L(\Bbb{R},\mu)$} and its applications.
    \newblock {\em J. Symb. Log.}, 80(1):29--55, 2015.
    
    \bibitem[Woo]{GenCM}
    W.~Hugh Woodin.
    \newblock Generalized {C}hang models.
    \newblock To appear.
    
    \end{thebibliography}
\end{document}